\documentclass[10pt]{amsart}
\usepackage[utf8]{inputenc}
\usepackage[english]{babel}
\usepackage{amsmath}
\usepackage{amsfonts}
\usepackage{amssymb}
\usepackage{mathabx}
\usepackage{graphicx}
\usepackage{float}
\usepackage[usenames, dvipsnames]{color}
\usepackage{soul}
\usepackage{fullpage}
\usepackage{hyperref}

\numberwithin{equation}{section}
\newtheorem{thm}{Theorem}
\newtheorem{conjecture}[thm]{Conjecture}
\newtheorem{proposition}[thm]{Proposition}
\newtheorem{corollary}[thm]{Corollary}
\newtheorem{lemma}[thm]{Lemma}
\newcommand{\eff}{{\rm eff}}
\title{The random pseudo-metric on a graph defined via the zero-set of the Gaussian free field on its metric graph}
\author{Titus Lupu \and Wendelin Werner}
\address {
Institute for Theoretical Studies,
ETH Z\"urich,
Clausiusstr. 47,
8092 Z\"urich,
Switzerland}
\email
{titus.lupu@eth-its.ethz.ch}

\address{
Department of Mathematics,
ETH Z\"urich,
 R\"amistr. 101,
8092 Z\"urich, Switzerland}
\email
{wendelin.werner@math.ethz.ch}

\dedicatory {Dedicated to the memory of Marc Yor}

\begin{document}

\begin{abstract}

We further investigate properties of the Gaussian free field (GFF) on the metric graph associated to a discrete weighted graph (where the edges of the latter are replaced by continuous line-segments of appropriate length) that has been introduced by the first author. 
On such a metric graph, the GFF is a random continuous function that generalises one-dimensional Brownian bridges so that one-dimensional techniques can be used. 

In the present paper, we define and study the pseudo-metric defined on the metric graph (and therefore also on the discrete graph itself), where the length of a path on the 
metric graph is defined to be the local time at level zero accumulated by the Gaussian free field along this path.  

We first derive a pathwise transformation that relates the GFF on the metric graph with the reflected GFF on the metric graph via the pseudo-distance defined by the latter. This is a generalisation of Paul Lévy's result relating the local time at zero of Brownian motion to the supremum of another Brownian motion. We also compute explicitly the distribution of certain functionals of this pseudo-metric and of the GFF. 
In particular, we point out that when the boundary consists of just two points, the law of the pseudo-distance between them depends solely on the resistance of the network between 
them. 

We then discuss questions related to the scaling limit of this pseudo-metric in the two-dimensional case, which should be the conformally invariant way to measure 
distances between CLE(4) loops introduced and studied by the second author with Wu, and  by Sheffield, Watson and Wu. Our explicit laws on metric graphs also
lead to new conjectures for related functionals of the continuum GFF on fairly general Riemann surfaces.
\end{abstract}


\maketitle
\tableofcontents

\section{Introduction}
\label{SecIntro}

\subsection*{A general question} 
When one is given a weighted graph (or an electric network), what are the natural random metrics or pseudo-metrics that one can define on them? 
This simple question is of course not new and has given rise to a number of works, and in the continuum, the definition of random metrics is sometimes described in theoretical physics under the name of 
quantum gravity (see eg. the references in \cite {DuplantierSheffield}). In the present paper, we will define and describe a simple and fairly natural random pseudo-metric on an electric network, point out some of its nice
features, and briefly discuss some conjectures about scaling limits, in particular in the two-dimensional case. 

\subsection*{One-dimensional local time background}
When $(\beta_{t})_{t\geq 0}$ is a standard one-dimensional Brownian motion started from $0$ and 
$(L_t)_{t\ge 0}$ denotes its local time at zero, 
then it is well-known (it is one of the many results by Paul L\'evy about Brownian motion, see for instance Theorem 2.3, chapter VI, § 2 in \cite{RevuzYor1999BMGrundlehren}), 
that the process $(\vert \beta_{t}\vert, L_t)_{t\geq 0}$ is distributed exactly like 
$(W_t - I_t, - I_t)_{t \ge 0}$ where $W$ is a standard Brownian motion and $I_t := \min_{s \in [0,t]} W_s$.  It can be viewed as a direct consequence of Tanaka's formula
$d \vert \beta_{t}\vert = \hbox {sgn} (\beta_{t}) d \beta_{t} + d L_t$;
Indeed, this shows that $W_t := \vert \beta_{t}\vert - L_t$ is a Brownian motion and it is easy to check that $- L_t = \min_{[0,t]} W$. This is also closely related to the 
so-called Skorokhod reflection Lemma (see also \cite{RevuzYor1999BMGrundlehren}).

Before explaining our generalisation of this result to general metric graphs, let us first describe the  special case of the segment $[0,T]$ (the previous Brownian case  corresponds 
in fact to the case where $T$ is infinite), which is due to Bertoin and Pitman \cite {BertoinPitman}. Instead of Brownian motion, we now consider $(W_t)_{t \in [0,T]}$ to be 
a Brownian bridge with $W_0 =w_0$ and $W_T = w_T$, where $w_0$ and $w_T$ are fixed and non-negative. For $t \in [0,T]$, we then define $I_t := \max ( \min_{[0,t]} W, \min_{[t,T]} W )$ (see Figure \ref {BBridge}).
We also consider another Brownian bridge $\beta$ on $[0,T]$ from $w_0$ to $w_T$, that we use to define the ``reflected Brownian bridge'' $|\beta|$. We can then define the local time at $0$ process $L$ of this 
reflected bridge (or equivalently of $\beta$). 
For each $t \in [0,T]$, we define $\delta_t := \min ( L_t, L_T - L_t ) = \min ( L [0,t], L[t,T])$.  
Then (see Theorem 4.1  in \cite {BertoinPitman}),
the process $( |\beta_t| - \delta_t )_{ t \in [0,T]}$ is distributed exactly like $W$ and the process $(\min ( I_t, 0))_{t \in [0,T]}$ is distributed like $-\delta$. 
Hence, we see that L\'evy's theorem has also a more symmetric version, and does not rely on the orientation of the real line. 
In this setting, it is therefore possible to construct a (non-reflected) Brownian bridge $W$ 
from a reflected Brownian bridge $|\beta|$ by using the latter's local time at $0$. 

\begin{figure}[ht!]
  \includegraphics[width=.8\textwidth]{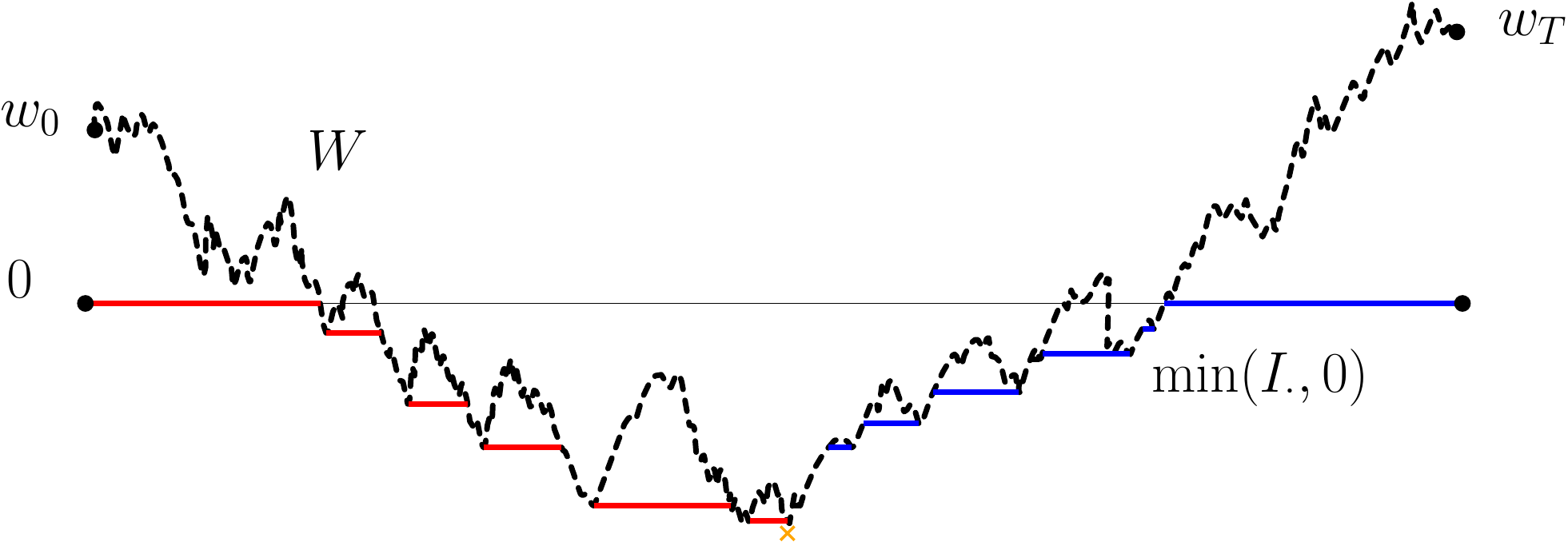}
  \vskip 5mm
  \caption{Sketch of a Brownian bridge $W$ (in dashed) and of the corresponding process $t \mapsto \min (I_t, 0)$ (in plain -- the two colors correspond to the portions before and after the minimum of the bridge). The generalisation of L\'evy's theorem by Bertoin and Pitman \cite {BertoinPitman} states that  the difference between the two is distributed like a reflected Brownian bridge $|\beta|$ from $w_0$ to $w_T$, and that conversely, one can recover $W$ from $|\beta|$. 
}
  \label{BBridge}
\end{figure}
Note that it is easy\footnote {Since this formula will be used throughout the paper, let us recall briefly a possible direct proof: For a Brownian motion $W$ started from $W_0$, one knows from the reflection principle \cite{RevuzYor1999BMGrundlehren} the joint law of $(W_T, I_T)$, from which one can deduce  the law of $(|W_T|, L_T)$ by L\'evy's theorem. We note that by reflection, the laws of $(W_T, L_T) 1_{L_T >0}$ and of $(-W_T, L_T) 1_{L_T >0}$ coincide while $L_T = 0$ implies that $W_T$ has the same sign as $W_0$. From this, one can deduce the joint law of $(W_T, L_T)$, and by conditioning on the value of $W_T$, one gets the law of $L_T$ for the Brownian bridge.} to compute the law of $L_T$, 
\begin{equation}
\label{EqLocTimeZeroBridge}
\mathbb{P}(L_T > \ell )=
\exp \left(-\frac{1}{2T}(\vert w_0 \vert+\vert w_T \vert + \ell )^{2}
+\frac{1}{2T}(w_0 -w_T)^{2}\right)
\end{equation}
for all $\ell \ge  0$ (see for instance \cite{BorodinSalminen1996HandbookBM}, page 155, Formula 1.3.8). 
Note that this formula holds regardless of the signs of $w_0$ and $w_T$. 
In particular, when $w_0$ and $w_T$ have the same sign, then the probability that $L_T$ is positive is 
$\exp(-2w_0 w_T /T )$ (which is the probability that the bridge touches $0$),  but $L_T$ is almost surely positive  when $w_0$ and $w_T$ have opposite signs.
A very  special case is  when $w_0 = w_T =0$, where $L_T^2 / (2T)$ is distributed like an exponential variable with mean $1$. 

\subsection*{GFF on metric graphs background} 
When one is given a finite connected undirected graph ${\mathcal G}$ consisting of a set of vertices $V$, a set of edges $E$ where each edge $e$ is also equipped with a positive and finite conductance $C(e)$, one can define 
a natural continuous-time Markov chain (that we will refer to as the continuous-time random walk) on the graph, and the related notion of harmonic functions. If one is given a non-empty subset $A$ of $V$, one can view $A$ as a boundary and one can then define the GFF with zero-boundary conditions on $A$ to be the centered Gaussian process on $V$, with covariance given by the Green's function of the random walk killed upon reaching $A$. 
More generally, when $h$ is a given function on $A$, one can define the GFF $\phi$ with boundary conditions $h$ on $A$ to be the sum of the aforementioned GFF with zero boundary conditions on $A$, with the deterministic harmonic extension of $h$ to $V$ (i.e., the function that is equal to $h$ on $V$ and is harmonic at all other vertices).

As explained and used in \cite{Lupu2014LoopsGFF,Lupu2015ConvCLE}, it can be very useful to respectively couple random walks and the GFF on such a discrete graph with Brownian motion and the GFF on the metric graph that is naturally associated to it (in particular when relating the GFF to the so-called loop-soups on those graphs).  
This metric graph is the structure $\widetilde {\mathcal G}$ that one obtains when one replaces formally each edge $e$ by a one-dimensional segment of length $R(e)=1/C(e)$. A point $x$ in the metric graph is 
therefore either a vertex in $V$ or a point on one of these segments. One can then easily define Brownian motion on this metric graph (loosely speaking, it moves 
like one-dimensional Brownian motion on the segments, and when at a vertex, it locally chooses to do excursions in each of the incoming segments uniformly). The trace of this Brownian motion on 
the sites of the graph (when parametrised by its local time at these sites) is then exactly the continuous-time random walk on the discrete graph, and the Green's functions on the metric graph does coincide with that of this continuous-time random walk, when looked at the vertices in $V$. 

In particular, this shows that one can define the GFF $\tilde \phi$ on the metric graph (with boundary $A$) in two equivalent ways: Either directly, as the Gaussian process with covariance function given by the Green's function of the Brownian motion on the metric graph. The field $\tilde \phi$ is then a strong Markov field on $\widetilde {\mathcal G}$ (\cite{Lupu2014LoopsGFF}).
 Or alternatively, by first sampling the discrete GFF $\phi$ on the discrete graph $\mathcal G$, which provides the value of $\tilde \phi$ when restricted to the sites of $V$. Then one has, for each edge $e$ to ``fill in'' the values of the GFF, using independent Brownian bridges (the time-length of the bridge is the length of this interval, and the values of the bridges at the end-points are given by the values of the GFF on the vertices).

As pointed out in \cite {Lupu2014LoopsGFF}, considering the metric graph allows to describe the conditional law of the GFF $\tilde \phi$ given its absolute value $|\tilde \phi|$ (or equivalently given its square $\tilde \phi^2$, which is a quantity that is naturally connected to loop-soups, see \cite{LeJan2011Loops}). Indeed, one can easily make sense of the excursions away from $0$ by $|\tilde \phi|$ on the metric graph, and conditionally on $|\tilde \phi|$, the signs of these excursions are just i.i.d. This makes it  very natural to study further what one can do using the local time at $0$ of $|\tilde \phi|$, which is of course closely related to these excursions. This is the purpose of the present paper, which will illustrate the fact that on metric graphs, a number of identities 
and tricks related to one-dimensional Brownian motions, bridges or excursions (see \cite {BertoinPitman,YorABM}) can be adapted and shed some light on important features of the GFF in higher dimensions. 

\subsection*{Some results of the present paper}
Let us consider a real-valued function $h$ on the boundary $A$,
and the GFF $\tilde \phi$ on the metric graph with condition $h$ on $A$.
\begin{figure}[ht!]
  \includegraphics[width=.68\textwidth]{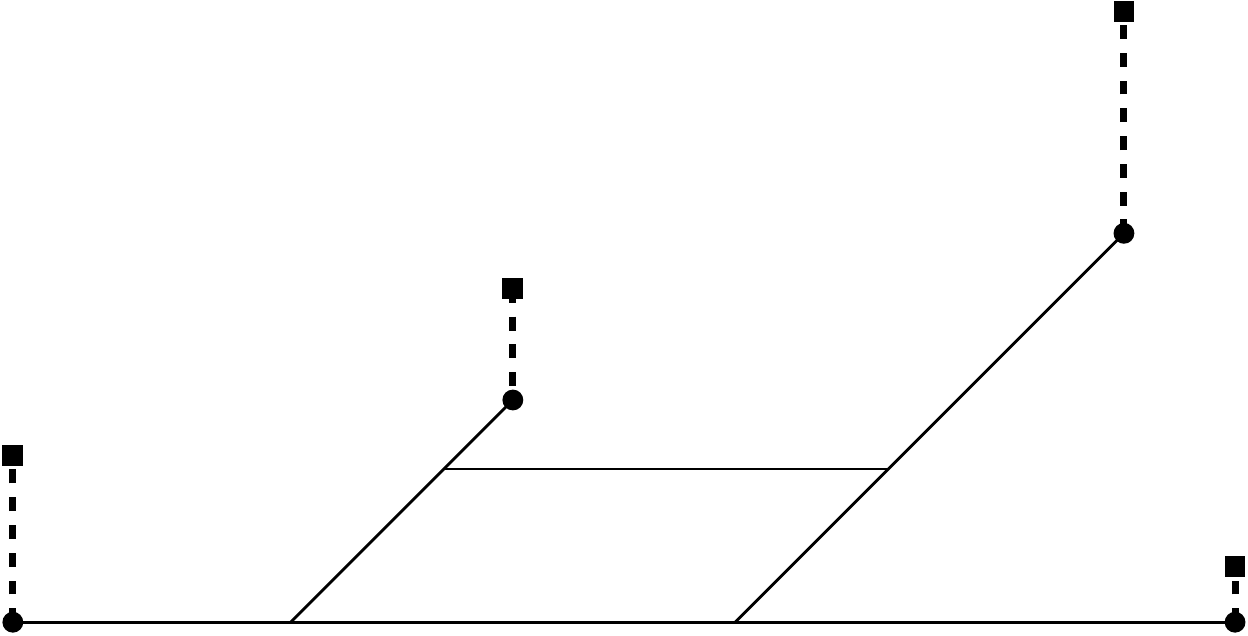}
  \caption{Sketch of a graph and positive boundary conditions at the four boundary points in $A$. {This is a picture in dimension $2+1$, and dashed lines and squares represent the height of the boundary conditions.}
}
  \label{GenLevy0}
\end{figure} 
On each edge, conditionally on the value of $\tilde \phi$ on the vertices, this process is a Brownian bridge, and it is therefore almost surely possible to define its local time at the level zero (on each of the edges).
We now use this local time to define the following pseudo-metric. Given $x$ and $y$ two points in the metric graph, the distance $\delta_{x,y}=\delta_{x,y}(\tilde{\phi})$ is the infimum over all continuous paths $\gamma$ joining $x$ to $y$ of the local time at $0$ accumulated along this path. Even though we will use the word ``distance'' throughout the paper, we stress that the mapping $(x,y) \mapsto \delta_{x,y}$ is in fact a pseudo-metric on the metric graph (and therefore also on the original graph): Two points are at distance zero from each other if and only if they 
belong to the same connected component of the non-zero set of $\tilde{\phi}$. For each $x$ in the metric graph, we consider the distance 
$\delta_{x,A}=\delta_{x,A}(\tilde{\phi})$ to the boundary set $A$ induced by the above pseudo-metric (again, mind that $\delta_{x,A}$ can be zero even if $x \notin A$).
Note that the quantities $\delta_{x,y}$ and $\delta_{x,A}$ can in fact be viewed as functionals of the absolute value of the GFF i.e. of the field $|\tilde \phi|$.

\begin{figure}[ht!]
    \includegraphics[width=.68\textwidth]{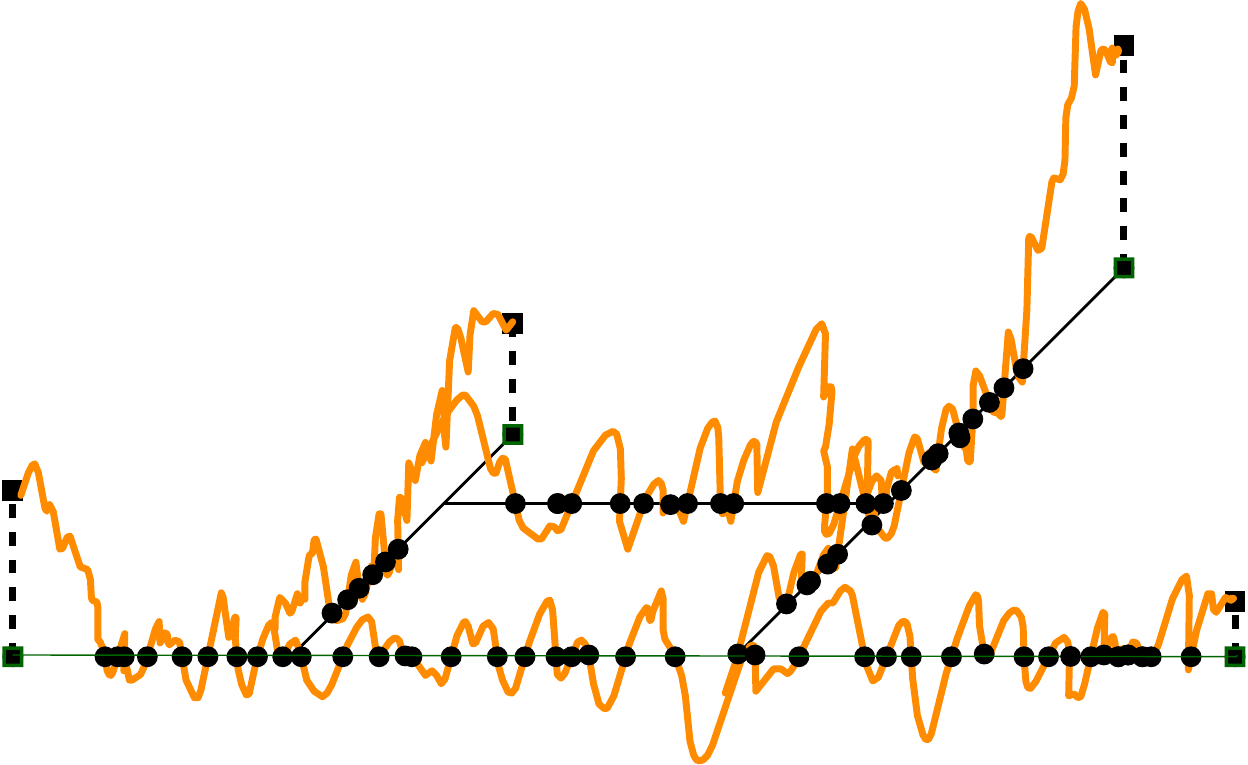} \vskip 5mm
  \caption{Sketch of a GFF and of the corresponding local time measure used to define the pseudo-distance $\delta$.}
  \label{GenLevy1}
\end{figure}

Similarly, when we are given the field $\tilde \phi$, we can define for each $x$ in the metric graph, the quantity $\tilde I_{x,A}$ to be the supremum over all 
continuous paths connecting $x$ to $A$ of the minimum of $\tilde \phi$ along this path, and then finally, $I_{x,A} = I_{x,A} ( \tilde \phi)  := \min (0, \tilde I_{x,A})$.
The processes $\delta_{x,A}$ and $I_{x,A}$ can be viewed as the metric graph generalisations of the local time process (of reflected Brownian motion and of the reflected Brownian bridge) and of the infimum process (of Brownian motion and Brownian bridges) that appear in L\'evy's result and its afore-mentioned generalisation to Brownian bridges.

\begin{figure}[ht!]
    \includegraphics[width=.68\textwidth]{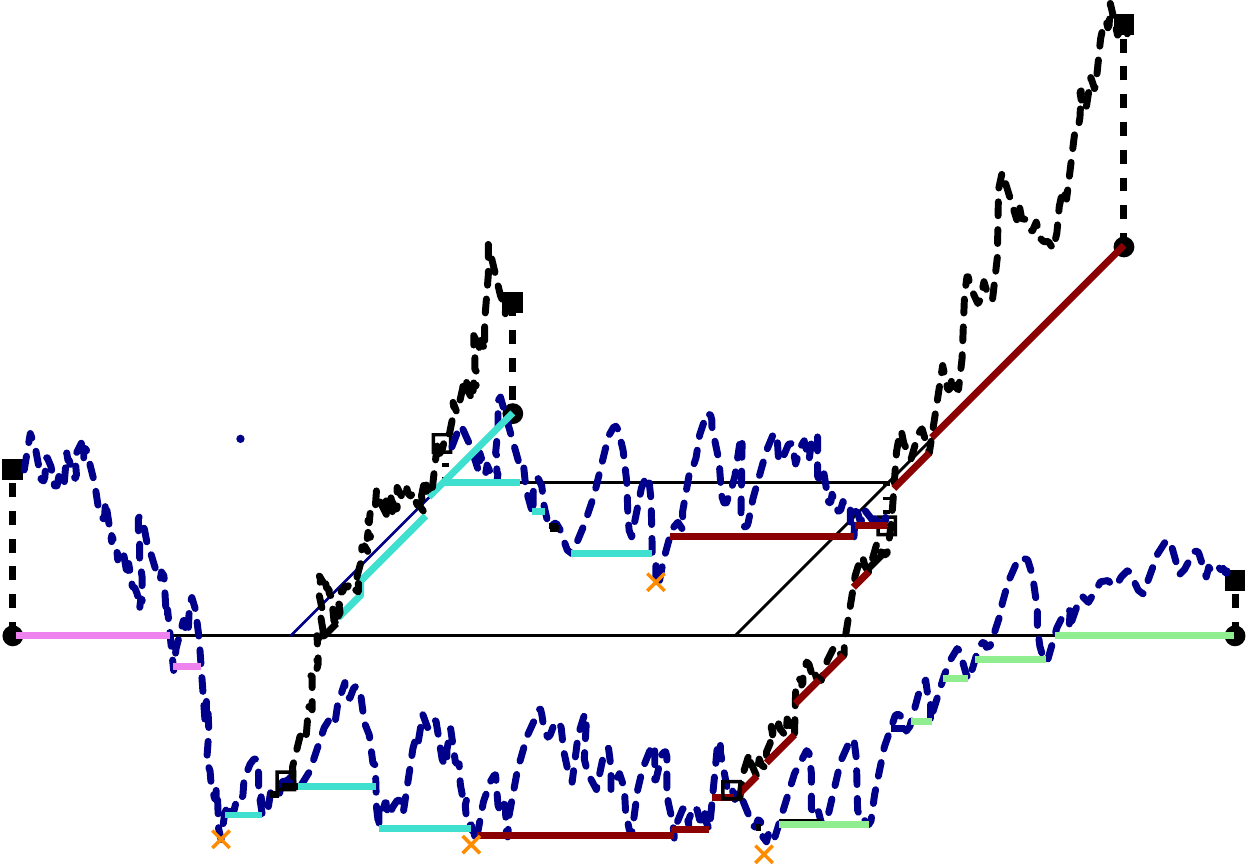} \vskip 5mm
  \caption{Sketch of a GFF (in dashed) with the function $x \mapsto I_{x,A}$ (in plain). Proposition \ref {thmGeneralisedLevy} says in particular that the difference between the two is a reflected GFF with the same boundary conditions, and that one can recover this picture from this reflected GFF. The different colors correspond to the ``basins of attraction'' of the four boundary points, and the crosses do separate these four regions.
}
  \label{GenLevy}
\end{figure}

The following generalisation of L\'evy's theorem to metric graphs will be the main result of Section \ref {S2} in the present paper:

\begin{proposition}
\label{thmGeneralisedLevy}
Consider a finite graph $\mathcal{G}$, a non-negative function $h$ defined on a non-empty subset $A$ of $V$, 
$\tilde \phi$ the metric graph GFF on $\widetilde{\mathcal{G}}$ with boundary condition $h$ on $A$, and define the processes $\delta$ and $I$ as before.
Then the  two pairs of processes
$(\vert\tilde{\phi}_{x}\vert,\delta_{x,A} )
_{x\in \widetilde{\mathcal{G}}}$ and $
(\tilde{\phi}_{x}-I_{x,A},
-I_{x,A} )_{x\in \widetilde{\mathcal{G}}}$
have the same law. 
\end{proposition}

In particular, this shows that the field
$(\vert\tilde{\phi}_{x}\vert-\delta_{x,A} )_{x\in \widetilde{\mathcal{G}}}$
has the same distribution as $\tilde{\phi}$. 

\medbreak

In Section \ref {S3}, will then point out features of some functionals of the pseudo-metric 
 that we now briefly illustrate: Consider a partition of the boundary $A$ into two non-empty sets
$\widehat{A}$ and $\widecheck{A}$, and suppose that the boundary condition $h$ has a constant sign on each of
$\widehat{A}$ and $\widecheck{A}$ (it does not have to be the same sign on both sets). 
Then we can define the pseudo-distance between $\widehat A$ and $\widecheck A$ 
to be the infimum over all paths that join a point of $\widehat A$ to 
a point in $\widecheck A$, of the cumulative local time along this path. It turns out that its law
is invariant if we modify the network $\mathcal{G}$ into an electrically equivalent circuit, seen simultaneously from all boundary points in $A$. A subcase of the 
general result (Proposition \ref {thmDist2Sets}) that will be given in Section \ref {S3} goes as follows:

\begin {proposition}
\label{PropLaw2Points}
 When $A$ contains only two points $\hat{x}$ and $\check{x}$, then the law of the
distance $\delta_{\hat{x},\check{x}}$ depends only on the {\em effective resistance} 
$R^{\rm eff}(\hat x,\check x)$
of the electric network between $\hat x$ and $\check x$.
This law is therefore given by Formula \eqref {EqLocTimeZeroBridge}, where $w_0 := h( \hat x)$, 
$w_T := h (\check x)$ and $T:= R^{\rm eff}(\hat x,\check x)$.
\end {proposition}

However, as we shall point out, it is easy to see that for the joint law of the distance between more than two boundary points, there is no such simple identity. 
Indeed, the laws of joint distances between three points turn out not to be 
invariant under local {star-triangle transformations (also known as Y-$\Delta$ transformations)} which do however preserve all resistances between boundary points.
 
It is also possible to use this pseudo-distance $\delta$ to define and study interesting random sets. 
Suppose that $h$, $\widehat A$ and $\widecheck A$ are as before. For each $a<\min_{\widehat{A}} h$, define
$\widetilde{\Lambda}_{a}$ to be the set of points $x$ in the metric graph, for which there exists a continuous path joining $x$ to $\widehat{A}$ such that
$\tilde{\phi}(z)\geq a$ along the path. 
Note that $\widetilde{\Lambda}_{a}$ is a random compact connected set containing $\widehat{A}$. 
We can call it \textit{the first {passage} set of level} $a$ from $\widehat A$. It is an
\textit{optional set} for the metric graph GFF $\tilde{\phi}$, that is to say that given $K$ a deterministic compact subset of $\widetilde{\mathcal{G}}$, the event $\widetilde{\Lambda}_{a}\subseteq K$ is measurable with respect the restriction of $\tilde{\phi}$ to $K$ (this is also closely related to the fact that 
$\widetilde{\Lambda}_{a}$ is a \textit{local set} for the GFF
$\tilde{\phi}$, see
\cite{SchrammSheffield2013ContourContGFF, AruSepulvedaWerner2016BoundedLocSets,Rozanov1982MarkovRandomFields}).

\begin{figure}[ht!]
  \centering 
  \includegraphics[width=.6\textwidth]{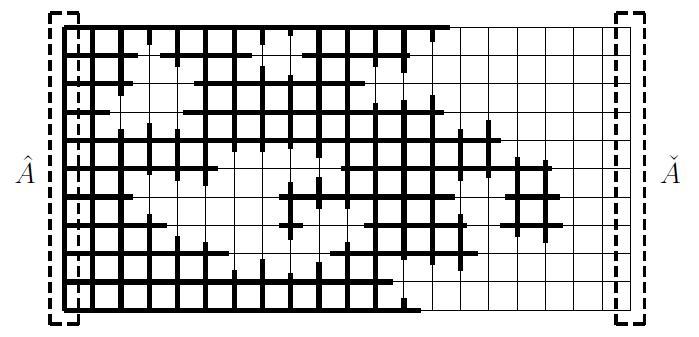}
  \caption{
Representation of $\widetilde{\Lambda}_{a}$, in thick lines. 
Boundary components are surrounded by dashed lines.
}
  \label{figLocalSet}
\end{figure}

We will study the distribution of the 
effective resistance $R^{\rm eff}(\widetilde{\Lambda}_{a},\widecheck{A})$ between the subsets
$\widetilde{\Lambda}_{a}$ and $\widecheck{A}$ in the electrical network $\widetilde{\mathcal{G}}$.
Note that this resistance is zero if and only if $\widetilde{\Lambda}_{a}$ intersects 
$\widecheck{A}$, and that only the connected components
of $\widetilde{\mathcal{G}}\setminus \widetilde{\Lambda}_{a}$ that intersect $\widecheck{A}$ matter to determine $R^{\rm eff}(\widetilde{\Lambda}_{a},\widecheck{A})$.
We will for instance  compute the Laplace transform of the effective conductance $1/ R^{\rm eff}(\widetilde{\Lambda}_{a},\widecheck{A})$ in the case where $h$ is constant on $\widecheck{A}$.

\subsection*{Further comments on the two-dimensional case}
While all the above considerations are not restricted to any particular dimensions (and in a way, this is one punchline of the present setup), one of the motivations for the present work comes from the very special two-dimensional case and the study of the GFF in 
continuous two-dimensional domains. Recall that in two dimensions, conformal invariance is a powerful additional property of the GFF. A lot of geometric properties and features of the 
GFF have been recently obtained using SLE, and some of them turn out to be closely related to properties of Brownian loop-soups 
(see e.g. \cite {SchrammSheffield2013ContourContGFF,SheffieldWerner2012CLE, MillerSheffieldCLE4GFF,MillerSheffieldIG1,QW} and the references therein). 
These results have been essentially obtained 
for the GFF in simply connected domains (where SLE tools are more easily applicable), and by local absolute continuity, they also provide some results (existence of level lines etc) on more general surfaces. Explicit formulas (for probabilities of events for instance) have also been produced in simply connected domains, using SLE techniques, but these appear 
difficult to extend to more general surfaces. 

The GFF on metric graphs approach has recently been shown to provide an alternative and useful approach to level lines of the continuum GFF in two dimensions (see \cite {Lupu2015ConvCLE} and also \cite {QW}).

The results of the present paper do involve quantities (typically the effective resistance) that have nice conformally invariant counterparts in the appropriate scaling limits. 
Indeed, taking the formal scaling limit of our computations in the simply connected case leads to known results for the continuous GFF. In Section \ref {S4}, we will: 
\begin {itemize}
 \item Briefly review the relation between the GFF on metric graphs, loop-soups and the conformal loop ensembles CLE$_4$ and the continuous GFF. 
 \item Explain why the scaling limit of our pseudo-distance should be the distance between CLE$_4$ loops that has been recently introduced and studied by Sheffield, Watson and Wu \cite {SheffieldWatsonWuMetricCLE4} and by Werner and Wu \cite {WernerWu2013ExplorCLE}, and why this could provide an alternative approach to some of its properties. 
 \item Give examples of formulae that conjecturally hold for the GFF (and the scaling limit of the pseudo-distance) on bordered Riemann surfaces, that seem out of the scope of SLE techniques. Here, \textit{extremal distance} (see \cite{Ahlfors2010ConfInv}, Section 4, and \cite{Duffin1962ELRes}) is (not surprisingly) the natural conformal invariant that appears in the scaling limit to describe these conjectural laws.  
\end {itemize}

\section{Generalisation of Lévy's theorem to metric graph GFF}
\label{SecLevy}
\label{S2}

\subsection {Approach by discretisation}

In this section we will establish Proposition \ref{thmGeneralisedLevy}. We will do this with a stochastic calculus-type proof, as 
will also smoothly lead to our further results. However, in the present subsection, we first describe informally another possible proof, based on a discretisations. One goal of this informal 
discussion is to explain that the generalisation of L\'evy's theorem should by no means be thought of as surprising (however, as the reader will probably realise, turning this informal description 
into a rigorous proof would require some technical work). Another goal is to introduce some of the ideas (such as simultaneous explorations of the values of the field when starting from the boundary) that will be used in the continuous setting. 

Let us briefly recall the following approach to L\'evy's theorem via discretisation. When $(S_n, n \ge 0)$ denotes a simple random walk, one can define the hitting times $\tau_{-m}$ of negative integers $-m$. 
Then, one can decompose the path of $S$ into i.i.d. pieces $(U^m, m \ge 1)$, where $U^m$ denotes the portion of $S + m$ between $\tau_{1-m}$ and $\tau_{-m}$. They are distributed each like a simple random walk started from $1$ and stopped at their 
first hitting of $0$. Hence, if we define $\widetilde U^m$ to be $U^m$ with an additional first upwards step from $0$ to $1$, and then concatenates these $U^m$, one gets exactly a reflected random walk  (i.e., distributed like $|S|$). Conversely,
starting from this reflected random walk, one can easily recover $S$ by erasing all its upwards moves from level $0$ to level $1$, and then concatenating the obtained pieces. Hence, in the discrete case, one has in fact a bijective transformation between the 
reflected random walk and the random walk. When one lets the number of steps grow to infinity (and renormalises appropriately, so that the random walk approaches a Brownian motion), then it is not difficult to see that this transformation 
converges to the identity in L\'evy's theorem (indeed, the appropriately renormalised number of visits to the origin by the reflected random walk will converge to the local time at the origin of the limiting Brownian motion). 

If one now wants to use this idea to derive the result for the one-dimensional Brownian bridge on a single segment $[0,T]$ that we described in the introduction, one can use the very same idea, ``discovering'' the random walk bridge from 
its two ends simultaneously. This time, the different pieces of the random walk will not be independent anymore, but one still has a one-to-one correspondence between a random walk bridge and a reflected random walk bridge. 
A new apparent difficulty arises because the total number of steps of the random walk bridge (say $2N$) is not the same as the total number of steps of the reflected bridge (indeed, one has added a certain number of forward upward 
moves from $0$ to $1$, and of the same number of backward upward moves from $0$ to $1$). However, when $N \to \infty$, this number will be of order $N^{1/2}$, and one can therefore resolve this problem by first choosing $N$ uniformly at random in $[n , n + n^{2/3} ]$ and then let $n$ tend to infinity:  
In this way, one has, for each $n$, a bijective correspondence between a random walk bridge of length $2N$, and a random path, which (with a total variation probability $1+ o(1)$ as $n \to \infty$) is close 
to a reflected random walk bridge of length $2N'$ (where $N'$ is distributed like $N$). Letting then $n$ tend to infinity, and using the invariance principle (and the local central limit theorem, so that the discrete 
bridges converge to Brownian bridges), one can fairly easily derive the generalisation of L\'evy's theorem in the case of a single segment. 

A generalisation to the case of metric graphs as stated in Proposition \ref {thmGeneralisedLevy} can be  obtained along similar lines. One divides for instance each segment  $e$ of the metric graph into a large number of 
discrete edges chosen uniformly in $[R(e) n, R(e) n + n^{2/3} ]$, with the constraint that every cycle in the graph must have even length. The approximation of the Gaussian field on the metric graph, by the (renormalised) discrete integer-valued field on 
this discrete graph is then the uniform measure on all integer-valued functions $f(x)$ on the discretised graph, with the constraint that the values of $f$ at any two neighbouring points always differ exactly by one, and with prescribed even values at the 
boundaries. The local central limit theorem also ensures that (when $n$ becomes large), the probabilities of two given possible outcomes does in fact not depend much on the exact outcome of the number of edges. 

Then, again, by discovering the values of this discrete field starting from the boundary, one gets for each $n$ a bijection between the pair $($discretisation of the graph, discrete field$)$ and  a pair given by 
$($discretisation of the graph, reflected discrete field$)$, 
where the latter is (in absolute variation) close to the former where one just changes the field into its absolute value. 
Letting $n$ tend to infinity, and using the invariance principle, the local central limit theorem as well as the fact, one can obtain Proposition~\ref {thmGeneralisedLevy}. 
Of course, we have omitted here quite a number of technical steps in this informal description.

\subsection {Some notations and preliminary remarks}

In the present subsection, we make a few elementary remarks concerning electric networks and their relations to GFF, and we also put down some important notation. 
When one considers a finite electric network (i.e. here, only a graph where the edges $e$ have conductances $C(e) = 1 / R(e)$), and one chooses two vertices $x_1$ and $x_2$, then one can define 
the equivalent or effective conductance $C^{\rm eff} (x_1, x_2)$ of the network between these two points. This means that for all purposes that do only involve the quantities related to $x_1$ and $x_2$, the network behaves exactly like the network with just one edge with conductance $C^{\rm eff} (x_1, x_2)$ joining $x_1$ to $x_2$ (see eg. \cite {LyonsPeres} for background on electric networks and their relations to random walks).

When we are now given a set $F$ of vertices in the graph that contains more than two elements, it is also easy to see that there exists a unique symmetric matrix $(C(x,y))$ indexed by $(x,y)$ in $F \times F$, such that for all purposes that involve quantities related to the vertices in $F$, the network behaves exactly like the network consisting of the sites $F$ and edges joining 
these sites in such a way that the conductance of the edge between $x$ and $y$ is $C(x,y)$ (with the constraint that $C(x,x)$ is always zero). We will refer to this matrix as the effective 
conductance matrix of $F$ in this network, and we will denote it by $(C_F^{\rm eff} (x,y))$. It should be stressed that in general, the quantity $C^{\rm eff}_F (x,y)$ is not equal 
to the previous ``two-point'' effective conductance $C^{\rm eff} (x,y) = C^{\rm eff}_{\{ x,y\}} (x,y)$.

Some readers may actually probably prefer to think of the matrix $C_F^{\rm eff}$ as a boundary excursion kernel (and this interpretation will indeed be useful later). Indeed, one can define the natural Brownian excursion measure $\mu_F$ away from $F$ in the metric graph as the limit when $\varepsilon \to 0$ of the sum over all points $y$ located exactly at distance $\varepsilon$ of $F$, of $\varepsilon^{-1}$ times the law of Brownian motion on the metric graph, started at $y$ and stopped at its first hitting of $F$. This is then an infinite measure on Brownian paths in the metric graph, that start on $F$ and end on $F$, and stay away from $F$ during the interior of their life-time. The mass of the excursions that start and end at the same point $x \in F$ is infinite, but the mass of the excursions that start at $x$ and end at another point $y \neq x$ in $F$ is finite. It is an easy exercise (one way would be to notice that it is obvious in the case where there are no other vertices than $F$ in the graph, and that this mass does not change under local star to complete graph transformations) to check that this mass is in fact 
exactly $C^{\rm eff}_F (x,y)$. Hence, $C^{\rm eff}_F (x,y)$ is in fact really the metric graph version of the boundary Poisson excursion kernel in the metric graph with boundary $F$, and the notation 
{$H_F (x,y)$} could be equally appropriate for it.

This effective conductance matrix is very closely related to the GFF. For instance, one can consider the previous electric network with a non-empty boundary $A$, and a 
real-valued function $h$ defined on $A$. It is then possible to define the GFF $\tilde \phi$ in the graph, with boundary condition $h$ on $A$. 
We then consider another finite set $B = \{ z_1, \ldots, z_n \}$ of other distinct points in the graph, and we would like to describe the joint law of $\tilde \phi$ on $B$. Then, it is easy to see that the Gaussian density of 
$(\tilde \phi (z), z \in B)$ will be described via the effective conductance matrix $C_{A \cup B}^{\eff}$, and more precisely, that it will be proportional to 
\begin {equation}
 \label {partitionfunction}
\prod_{x \in A, z \in B }\exp \left(  {C^{\rm eff}_{A \cup B} ( x,z)} \frac { ( \tilde \phi (z) - h(x))^2} 2 \right) \times \prod_{1 \le i < j \le n }  
\exp \left(  {C^{\rm eff}_{A \cup B} ( z_i,z_j)} \frac {( \tilde \phi (z_i) - \tilde \phi (z_j) )^2}2 \right).
\end {equation}

We are now going to  use natural ways to discover the GFF on the metric graph, by exploring it starting from part of the boundary.
This strategy is again reminiscent of the Schramm-Sheffield theory of local sets  as in \cite {SchrammSheffield2013ContourContGFF, MillerSheffieldIG1, MillerSheffieldQLE}. But of course, in the present metric graph setting, one is in the realm of usual stochastic calculus. 
The idea will again be to consider multiple particles evolving on the metric graph, all starting from the boundary, in such a way that
at any given time, all particles  will be at the same $\delta(\tilde{\phi})$ distance from the boundary.
The particles will progressively discover the GFF. We will write the stochastic differential equation for
$\tilde{\phi}_{\bullet}$ and for 
$\vert\tilde{\phi}_{\bullet}\vert-\delta_{\bullet,A}(\tilde{\phi})$, 
were $\bullet$ stands for the multiple evolving particles, and check that it is the same in both cases.

Let us briefly recall the setup: 
We consider a finite graph $\mathcal{G}=(V,E)$ and $A$ a non-empty subset of $V$. We suppose that $E$ contains no edges that joins a site of $A$ directly to another site of $A$. 
Let $e_{1},\dots,e_{n}$ be the edges adjacent to $A$ i.e. that connect a site of $A$ to a site in $V\setminus A$. $x_{i}$ will denote the endpoint of $e_{i}$ in $A$, $y_{i}$ the endpoint in $V\setminus A$. 
For $r_{i}\in[0,R(e_{i})]$,
$z_i := z_{i}(r_{i})$ will denote the point in the edge $[x_{i},y_{i}]$ lying at distance $r_{i}$ from $x_{i}$.
For this collection $r_i$, we define the set $B:= \{ z_i, 1 \le i \le n\}$, and we look at the effective conductance matrix of $A \cup B$. 
The construction shows immediately, that $C^{\eff}_{A \cup B} (z_i, z_j)$ for $i < j$ depends only on the new electric network that is obtained by simply erasing all the 
portions $[x_i, z_i]$ (and it is then just $C^{\rm eff}_B$ for this network). 
When $A$ is considered to be fixed once and for all, we will write  
$$ C_{ij}^{\rm eff} ( r_1, \ldots, r_n ) := C^{\eff}_{A \cup B} (z_i, z_j).$$
This means that one can first discover the electric quantities on the network on the edges $[x_i, z_i]$ and then rewire the not-yet explored part of the network and replace it by a 
weighted complete graph joining the $z_i$'s, as illustrated in Figure \ref {MetricTLRewiring}.  

\begin{figure}[ht!]
  \centering 
  \includegraphics[width=.7\textwidth]{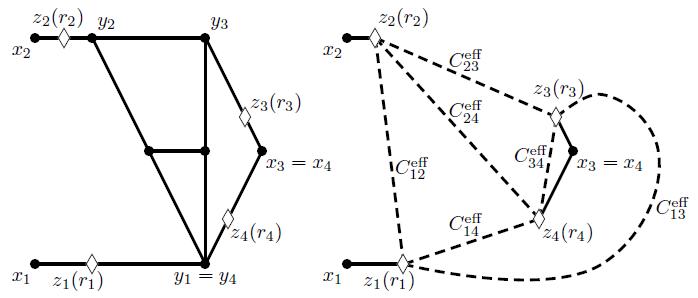}
  \caption{
The original network $\widetilde{\mathcal{G}}$ on the left and the network after rewiring on the right.
Dashed lines are the edges used for rewiring, with corresponding effective conductances. 
}
  \label{MetricTLRewiring}
\end{figure}

Let us now explain how to use this procedure also to discover the GFF $\tilde \phi$ first on the edges $[x_i, z_i]$, and also to describe the marginal law of the GFF on these edges. 
Given $r_{1},\dots,r_{n}$, such that $r_{i}\in [0,R(e_{i})]$, let us define
$$\tilde \phi (r_1, \ldots, r_n) := ( ( \tilde{\phi}_{z_{1}(r)})_{ r \leq r_{1}}, \ldots, (\tilde{\phi}_{z_{n}(r)})_{ r \leq r_{n}} )$$
and let  $\mathcal{F}(r_{1},\dots,r_{n})$ be the $\sigma$-algebra that generated by this process.

The following lemma is then an immediate consequence  of (\ref{partitionfunction}) and of the spatial Markov property of the GFF: 
\begin{lemma}
\label{LemAbsCont}
When $r_{i}\in [0,R(e_{i}))$, the law of $\tilde \phi (r_1, \ldots, r_n)$ is absolutely continuous with respect to
the law of $n$ independent Brownian motions $(W^1, \ldots, W^n)$ respectively started from $(h(x_1), \ldots, h(x_n))$ 
and defined on the time intervals $[0,r_1], \ldots, [0,r_n]$. 
The relative intensity between these two distributions at $w=(w^{1} ( \cdot), \ldots , w^{n} ( \cdot))$ is equal to 
\begin{equation}
\label{EqDensityConduc}
D(r_{1},\dots,r_{n}) [w] :=
\dfrac{1}{Z(r_{1},\dots,r_{n})}
\prod_{1\leq i<j\leq n}\exp\left(-C^{\rm eff}_{ij}(r_{1},\dots,r_{n})
\frac {(w^{(j)}(r_{j})-w^{(i)}(r_{i}))^{2}} 2\right),
\end{equation}
{where $Z(r_{1},\dots,r_{n})$ is the corresponding partition function so that the expected value of $D(r_1, \ldots , r_n)[W]$ when $W=(W^1, \ldots, W^n)$ is equal to $1$.}
\end {lemma}
In the sequel, we will sometimes simply write $D(r_1, \ldots, r_n)$ as a shorthand for the random variable $D(r_{1},\dots,r_{n})[\tilde \phi(r_1, \ldots, r_n)]$. 
With this notation, given $r_{1},\dots,r_{n}$ and $r'_{1},\dots,r'_{n}$ such that 
$0\leq r_{i}\leq r'_{i}< R(e_{i})$, we can iterate the procedure and see that: 
\begin{equation}
\label{EqDensityMultiMart}
\mathbb{E}[1 / D(r'_{1},\dots,r'_{n}) \vert \mathcal{F}(r_{1},\dots,r_{n})]=
1/ D(r_{1},\dots,r_{n}).
\end{equation}

Note that when $r_{i}=R(e_{i})$, it may happen that $y_{i}=y_{j}$ for some $i \neq j$. Then
$\tilde{\phi}_{z_{i}(r_i)}=\tilde{\phi}_{z_{j}(r_j)}$ and there is no absolute continuity anymore with the law of $n$ independent Brownian motions.

We will consider the class of continuous non-decreasing stochastic
processes $(r_{i}(t))_{1\leq i\leq n, t\geq 0}$, coupled with the metric graph GFF $\tilde{\phi}$, such that
$r_{i}(0)=0$, $r_{i}(t)\in [0,R(e_{i})]$, and
for all $(r_{1},\dots,r_{n})\in\prod_{i=1}^{n}[0,R(e_{i})]$ and $t\geq 0$, 
$$\lbrace\forall i\in\lbrace 1,\dots,n\rbrace, r_{i}(t)\leq r_{i}\rbrace 
\in \mathcal{F}(r_{1},\dots,r_{n}).$$
We will refer to such processes 
$(r_{i}(t))_{1\leq i\leq n, t\geq 0}$ as  \textit{explorations} of the metric graph GFF. Mind that we explore here only the edges adjacent to $A$, and that 
we also allow the exploration to go all the way to $r_i = R(e_i)$. 

The previous lemma allows easily to associate simple martingales to such explorations: Define for each $t \ge 0$,
$$
\tilde{\phi}_{i}(t):=\tilde{\phi}_{z_{i}(r_{i}(t))},
 \quad 
C^{\rm eff}_{ij}(t):=C^{\rm eff}_{ij}(r_{1}(t),\dots,r_{n}(t))
$$
and denote the natural filtration of 
$((\tilde{\phi}_{i}(t),r_{i}(t))_{1\leq i\leq n})_{t\geq 0}$
by $(\mathcal{F}_{t})_{t\geq 0}$.

\begin{lemma}
\label{LemSDE}
The processes 
$(\tilde{\phi}_{i}(t))_{1\leq i\leq n, t\geq 0}$ have the following semi-martingale decomposition:
\begin{equation}
\label{EqSDEExplor}
\tilde{\phi}_{i}(t)=
M_{i}(t)+\sum_{j\neq i}\int_{0}^{t}
C^{\rm eff}_{ij}(s)(\tilde{\phi}_{j}(s)-\tilde{\phi}_{i}(s)) dr_{i}(s),
\end{equation}
where $M_{i}$ is a continuous martingale with respect the filtration of
$(\mathcal{F}_{t})_{t\geq 0}$. Moreover, the quadratic variations satisfy
\begin{equation}
\label{EqVarQuad}
\langle M_{i},M_{i}\rangle_{t}=r_{i}(t),
\qquad \langle M_{i},M_{j}\rangle_{t}=0,~\text{for}~i\neq j.
\end{equation}
\end{lemma}

\begin{proof}
Let $\varepsilon>0$ small.
Let us first assume that for all $i\in \lbrace 1,\dots,n\rbrace$ and $t\geq 0$, 
$r_{i}(t)\leq R(e_{i})-\varepsilon$.
Denote, with same the notation as in \eqref{EqDensityMultiMart},
\begin{displaymath}
D_{t}:=D(r_{1}(t),\dots,r_{n}(t)).
\end{displaymath}
Let $(\widetilde{M}_{i}(t),\tilde{r}_{i}(t))_{1\leq i\leq n, t\geq 0}$ be the process 
with the density $D^{-1}(R(e_{1})-\varepsilon,\dots,R(e_{n})-\varepsilon) [\cdot]$ with respect the law of
$(\tilde{\phi}_{i}(t),r_{i}(t))_{1\leq i\leq n, t\geq 0}$.
Let $(\widetilde{F}_{t})_{t\geq 0}$ be the natural filtration of
$(\widetilde{M}_{i}(t),\tilde{r}_{i}(t))_{1\leq i\leq n, t\geq 0}$.
According to Lemma \ref{LemAbsCont}, there are $n$ independent Brownian motions 
$(W^{(i)}(r))_{1\leq i\leq n, 0\leq r\leq R(e_{i})}$,
such that
\begin{displaymath}
\widetilde{M}_{i}(t)=W^{(i)}(\tilde{r}_{i}(t)).
\end{displaymath}
By construction, an event
\begin{displaymath}
\lbrace\forall i\in\lbrace 1,\dots,n\rbrace, \tilde{r}_{i}(t)\leq r_{i}\rbrace
\end{displaymath}
is measurable with respect to 
$(W^{(i)}_{r})_{1\leq i\leq n, 0\leq r\leq r_{i}}$.
Thus, the processes $(\widetilde{M}_{i}(t))_{1\leq i\leq n, t\geq 0}$ are
$(\widetilde{F}_{t})_{t\geq 0}$-martingales with
\begin{displaymath}
\langle \widetilde{M}_{i},\widetilde{M}_{i}\rangle_{t}=\tilde{r}_{i}(t),
\qquad
\langle \widetilde{M}_{i},\widetilde{M}_{j}\rangle_{t}=0,~\text{for}~i\neq j.
\end{displaymath}

Using Lemma \ref {LemAbsCont}, \eqref{EqDensityMultiMart} and dyadic partitions, one can show that for any positive time $t$, the density of 
the process $(\widetilde{M}_{i}(s),\tilde{r}_{i}(s))_{1\leq i\leq n, 0\leq s\leq t}$ with respect to
$(\tilde{\phi}_{i}(s),r_{i}(s))_{1\leq i\leq n, 0\leq s\leq t}$ is $D^{-1}_{t}$.
According to \eqref{EqDensityConduc}, we can write
\begin{displaymath}
D_{t}=\dfrac{1}{Z(\tilde{r}_{1}(t),\dots,\tilde{r}_{n}(t))}
\prod_{1\leq i<j\leq n}\exp\left(-\dfrac{1}{2}C^{\rm eff}_{ij}(\tilde{r}_{1}(t),\dots,\tilde{r}_{n}(t))
(\widetilde{M}_{j}(t)-\widetilde{M}_{i}(t))^{2}\right).
\end{displaymath}
The function $Z(\tilde{r}_{1}(t),\dots,\tilde{r}_{n}(t))$ has bounded variation, so that It\^o's formula implies that
\begin {displaymath}
\left\langle  \dfrac{d D_{t}}{D_{t}},d\widetilde{M}_{i}(t)\right\rangle=
\sum_{j\neq i}C^{\rm eff}_{ij}(\tilde{r}_{1}(t),\dots,\tilde{r}_{n}(t))(\widetilde{M}_{j}(t)-\widetilde{M}_{i}(t)) dr_{i}(t).
\end{displaymath}
The lemma in this particular case follows by applying Girsanov's theorem (for the multidimensional version of Girsanov's Theorem,
see for instance \cite{KaratzasShreve2010BMStochCalc}, Section 3.5).

Let us now consider the general case where we do  not assume anymore that $r_{i}(t)\leq R(e_{i})-\varepsilon$. 
Let $\tau_{1}^{\varepsilon}$ be the first time at which one of the $r_{i}(t)$
hits $ R(e_{i})-\varepsilon$. Then,
$(\tilde{\phi}_{i}(t\wedge \tau_{1}^{\varepsilon} ))_{1\leq i\leq n, t\geq 0}$ satisfies 
\eqref{EqSDEExplor} and \eqref{EqVarQuad}.
As $\varepsilon$ tends to $0$, $(\tilde{\phi}_{i}(t\wedge \tau_{1}^{\varepsilon} ))_{1\leq i\leq n, t\geq 0}$
converges to $(\tilde{\phi}_{i}(t\wedge \tau_{1}))_{1\leq i\leq n, t\geq 0}$, where
$\tau_{1}$ is the first time one of the $r_{i}(t)$ hits $ R(e_{i})$. 
The stopped martingales $(M_{i}(t\wedge \tau_{1}^{\varepsilon}))_{t\geq 0}$ have a bounded quadratic variation uniformly
in $\varepsilon$. Thus the stopped martingales converge as $\varepsilon$ tends to $0$.
Consequently the drift term in \eqref{EqSDEExplor} converges too.
By convergence,
$(\tilde{\phi}_{i}(t\wedge \tau_{1}))_{1\leq i\leq n, t\geq 0}$ satisfies 
\eqref{EqSDEExplor} and \eqref{EqVarQuad}. 

After time $\tau_{1}$, the $r_{i}(t)$-s that did already
hit $R(e_{i})$, as well as the corresponding $\tilde{\phi}_{i}(t)$-s, stay constant. The number of evolving processes $\tilde{\phi}_{j}(t)$ decreases but the evolution equation is the same. Indeed, we can apply the argument of absolute continuity with respect a family of time-changed independent Brownian motions to a smaller number of processes. We can decompose the time-line $[0,+\infty)$ using stopping times 
$\tau_{0}=0< \tau_{1}<\dots< \tau_{k}< \tau_{k+1}=+\infty$, where $\tau_{1},\dots,\tau_{k}$ correspond to the successive times when 
one of the $r_{i}(t)$ hits $R(e_{i})$. On each interval $[\tau_{q},\tau_{q+1})$, the processes 
$\tilde{\phi}_{j}(t)$ that are not frozen evolve
according \eqref{EqSDEExplor} and \eqref{EqVarQuad}. Since all the processes $\tilde{\phi}_{j}(t)$ are continuous, \eqref{EqSDEExplor} and \eqref{EqVarQuad} are valid on $[0,+\infty)$.
\end{proof}

The previous lemma allows to obtain the following simple variational property of the partition functions $Z(r_1, \ldots, r_n)$
{and of the effective conductances $C^{\rm eff}_{ij}(r_{1},\dots,r_{n})$}. Note that other proofs of these deterministic facts
are possible (this is indeed analogous to Hadamard's variational formula
\cite[Chapter 15]{Garaberdian1986PDE} which gives the first order variation of a Green's function under a perturbation of a continuum domain -- the Green's function corresponds to effective resistances).
\begin{lemma}
\label{LemDerivPartFunc}
Let $(r_{1},\dots,r_{n})\in\prod_{i=1}^{n}(0,R(e_{i}))$. Then, with the notation of
\eqref{EqDensityConduc}, we have 
\begin{equation}
\label{EqDerLogZ}
\dfrac{\partial_{r_{i}}Z(r_{1},\dots,r_{n})}{Z(r_{1},\dots,r_{n})}
+\dfrac{1}{2}
\sum_{j\neq i}
C^{\rm eff}_{ij}(r_{1},\dots,r_{n})
=0.
\end{equation}
For $j\neq i$,
\begin{equation}
\label{EqDerCond}
\partial_{r_{i}}C^{\rm eff}_{ij}(r_{1},\dots,r_{n})=
C^{\rm eff}_{ij}(r_{1},\dots,r_{n})\sum_{j'\neq i}C^{\rm eff}_{ij'}(r_{1},\dots,r_{n}).
\end{equation}
For $i \notin \{ j, j'\}$, and $j \neq j'$,
\begin{equation}
\label{EqDerCond2}
\partial_{r_{i}}C^{\rm eff}_{jj'}(r_{1},\dots,r_{n})=
-C^{\rm eff}_{ij}(r_{1},\dots,r_{n})C^{\rm eff}_{ij'}(r_{1},\dots,r_{n}).
\end{equation}
\end{lemma}

\begin{proof}
By symmetry, we can restrict to the case where $i=1$.
With the notation of \eqref{EqDensityMultiMart},
the process $(D(r,r_{2},\dots,r_{n})^{-1})_{r\in (0,R(e_{1}))}$ is a martingale.
Applying Itô's formula and \eqref{EqSDEExplor} to it, we see that the drift term turns out to be equal to 
\begin{eqnarray*}
&&
D(r,r_{2},\dots,r_{n})^{-1}
\Big(\dfrac{\partial_{r_{1}}Z(r,r_{2},\dots,r_{n})}{Z(r,r_{2},\dots,r_{n})}+
\dfrac{1}{2}\sum_{j=2}^{n}C^{\rm eff}_{1j}(r,r_{2},\dots,r_{n})
\\
&+&
\sum_{j=2}^{n}\Big(\partial_{r_{1}}C^{\rm eff}_{1j}(r,r_{2},\dots,r_{n})
-C^{\rm eff}_{1j}(r,r_{2},\dots,r_{n})\sum_{j'=2}^{n}C^{\rm eff}_{1j'}(r,r_{2},\dots,r_{n})\Big)
(\tilde{\phi}_{z_{j}(r_{j})}-\tilde{\phi}_{z_{1}(r)})^{2}
\\
&+&
\dfrac{1}{2}\sum_{2\leq j<j'\leq n}
\Big(\partial_{r_{1}}C^{\rm eff}_{jj'}(r,r_{2},\dots,r_{n})
+C^{\rm eff}_{1j}(r,r_{2},\dots,r_{n})C^{\rm eff}_{1j'}(r,r_{2},\dots,r_{n})\Big)
(\tilde{\phi}_{z_{j'}(r_{j'})}-\tilde{\phi}_{z_{j}(r_{j})})^{2}\Big).
\end{eqnarray*}
Since this drift must be equal to zero, each pre-factor of a term
in $(\tilde{\phi}_{z_{j}(r_{j})}-\tilde{\phi}_{z_{1}(r)})^{2}$ or a
term in $(\tilde{\phi}_{z_{j'}(r_{j'})}-\tilde{\phi}_{z_{j}(r_{j})})^{2}$, as
well as the term not involving $\tilde{\phi}$, must be zero.
By identifying different terms, we get exactly
\eqref{EqDerLogZ}, \eqref{EqDerCond} and \eqref{EqDerCond2}.
\end{proof}

Finally, we note that conversely, the previous descriptions  can be used to actually construct a GFF. More specifically, 
let us consider a continuous stochastic process $(X_{i}(t),r_{i}(t))_{1\leq i\leq n, 0\leq t\leq T}$ and denote 
$(\mathcal{F}_{t})_{0\leq t\leq T}$ its natural filtration.
We further assume that:
\begin{itemize}
\item Each $r_{i}(t)$ is a non-decreasing process with values in $[0,R(e_{i})]$ and $r_{i}(0)=0$.
\item For each $i$, $X_{i}(0)=h(x_{i})$, and $X_{i}(t)$ is a semi-martingale with respect the filtration 
$(\mathcal{F}_{t})_{t\geq 0}$, with the semi-martingale decomposition
\begin{equation}
\label{EqSDEX}
X_{i}(t)=
M_{i}(t)+\sum_{j\neq i}\int_{0}^{t}
C^{\rm eff}_{ij}(r_{1}(s),\dots,r_{n}(s))(X_{j}(s)-X_{i}(s)) dr_{i}(s),
\end{equation}
where $M_{i}(t)$ is an $(\mathcal{F}_{t})_{0\leq t\leq T}$-martingale.
\item The martingales $M_{i}(t)$ satisfy
\begin{equation}
\label{EqSDEXQV}
\langle M_{i},M_{i}\rangle_{t}=r_{i}(t),
\qquad
\langle M_{i},M_{j}\rangle_{t}=0,~\text{for}~i\neq j.
\end{equation}
\end{itemize}
We can note that $X_{i}(t)$ is then constant on the time intervals where $r_{i}$ remains constant. 
We then define $(\hat{\phi}_{z})_{z\in\widetilde{\mathcal{G}}}$ 
to be the random field on the metric graph
$\widetilde{\mathcal{G}}$ coupled to $(X_{i}(t),r_{i}(t))_{1\leq i\leq n, 0\leq t\leq T}$ in the following way:
\begin{itemize}
\item For $t\leq T$, $\hat{\phi}_{z_{i}(r_{i}(t))}=X_{i}(t)$.
\item Outside the lines $[x_{1},z_{1}(r_{1}(T))],\dots,[x_{n},z_{n}(r_{n}(T))]$, conditionally on the value of
$\hat{\phi}$ on these lines, $\hat{\phi}$ is distributed like a metric graph GFF with boundary conditions 
$\hat{\phi}_{x_{1}}, \hat{\phi}_{z_{1}(r_{1}(T))},\dots,$
$\hat{\phi}_{x_{n}}, \hat{\phi}_{z_{n}(r_{n}(T))}$.
\end{itemize}
\begin {lemma}
\label {LemSDE2}
Under all these conditions, the field $\hat{\phi}$ is distributed like the metric graph GFF $\tilde{\phi}$.
\end{lemma}

\begin{proof}
Let $\varepsilon>0$ small. Let us first assume that for all $i$, $r_{i}(T)\leq R(e_{i})-\varepsilon$. 
We extend the definition of $r_{i}(t)$ and $X_{i}(t)$ for $t>T$,
\begin{displaymath}
r_{i}(t):=
(r_{i}(T)+(t-T))\wedge (R(e_{i})-\varepsilon) ~\text{if}~t>T,
\end{displaymath}
and set $X_{i}(t)=\hat{\phi}_{z_{i}(r_{i}(t))}$. Let $(\mathcal{F}_{t})_{t\geq 0}$ be the natural filtration
of $(X_{i}(t),r_{i}(t))_{1\leq i\leq n, t\geq 0}$.
From Lemma~\ref{LemSDE} follows that for $t\geq T$, the processes
$(X_{i}(t))_{1\leq i\leq n, t\geq T}$ satisfy the equations \eqref{EqSDEExplor} and \eqref{EqVarQuad}, because by definition, for $t\geq T$, $z_{i}(r_{i}(t))$ explores a metric graph GFF.
For $t\in[0,T]$, \eqref{EqSDEExplor} and \eqref{EqVarQuad} are satisfied too, because of
\eqref{EqSDEX} and \eqref{EqSDEXQV}. So for all $t\geq 0$, we have
\begin{displaymath}
X_{i}(t)=
M_{i}(t)+\sum_{j\neq i}\int_{0}^{t}
C^{\rm eff}_{ij}(r_{1}(s),\dots,r_{n}(s))(X_{j}(s)-X_{i}(s)) dr_{i}(s),
\end{displaymath}
where the quadratic variations and covariations of the $(\mathcal{F}_{t})_{t\geq 0}$-martingales 
$(M_{i}(t))_{1\leq i\leq n, t\geq 0}$ are given by \eqref{EqVarQuad}.
 
Define
\begin{displaymath}
D^{-1}_{t}:=Z(r_{1}(t),\dots,r_{n}(t))
\prod_{1\leq i<j\leq n}
\exp\left(\frac{1}{2}
C^{\rm eff}_{ij}(r_{1}(t),\dots,r_{n}(t))(X_{j}(t)-X_{i}(t))^{2}
\right).
\end{displaymath}
Applying Itô's formula and \eqref{EqDerLogZ}, \eqref{EqDerCond} and \eqref{EqDerCond2}, 
we  get that $(D^{-1}_{t})_{t\geq 0}$ is
an $(\mathcal{F}_{t})_{t\geq 0}$ local martingale. 
Moreover
\begin{displaymath}
D^{-1}_{0}=
Z(0,\dots,0)\prod_{1\leq i<j\leq n}
\exp\left(\frac{1}{2}C^{\rm eff}_{ij}(0,\dots,0)(h(x_{j})-h(x_{i}))^{2}\right).
\end{displaymath}
From Lemma \ref{LemAbsCont}, it follows that $D^{-1}_{0}=1$.

For $K>0$, denote $\theta_{K}$ the stopping time when one of the $X_{i}(t)$ hits $K$ or $-K$. 
The stopped local martingale $(D^{-1}_{t\wedge \theta_{K}})_{t\geq 0}$ is bounded, thus it is a true martingale. 
With the condition $D^{-1}_{0}=1$, we get that it is a change of measure martingale for the filtration 
$(\mathcal{F}_{t})_{t\geq 0}$. We apply the change of measure of density 
$(D^{-1}_{t\wedge \theta_{K}})_{t\geq 0}$ to the process
$(X_{i}(t\wedge\theta_{K}), r_{i}(t\wedge\theta_{K}))_{1\leq i\leq n, t\geq 0}$ 
and obtain a processes denoted 
$(\widetilde{M}_{i,K}(t), \tilde{r}_{i,K}(t))_{1\leq i\leq n, t\geq 0}$. By Girsanov's theorem,
$\widetilde{M}_{i,K}$ is a martingale with respect the natural filtration of
$(\widetilde{M}_{i,K}(t), \tilde{r}_{i,K}(t))_{1\leq i\leq n, t\geq 0}$ the quadratic variations are
\begin{displaymath}
\langle \widetilde{M}_{i,K},\widetilde{M}_{i,K}\rangle_{t}=\tilde{r}_{i,K}(t),
\qquad
\langle \widetilde{M}_{i,K},\widetilde{M}_{j,K}\rangle_{t}=0,~\text{for}~i\neq j.
\end{displaymath}
Let $\tilde{\theta}_{K,K'}$ be the first time one of the $\widetilde{M}_{i,K}$ hits $K'$ or $-K'$.
The processes do not evolve after time $\tilde{\theta}_{K,K'}$.
Let $K'\in(0,K)$. 
Then
$(\widetilde{M}_{i,K}(t), \tilde{r}_{i,K}(t))_{1\leq i\leq n, t\geq 0}$
and
$(\widetilde{M}_{i,K'}(t), \tilde{r}_{i,K'}(t))_{1\leq i\leq n, t\geq 0}$
are naturally coupled on the same probability space such that
\begin{displaymath}
(\widetilde{M}_{i,K'}(t), \tilde{r}_{i,K'}(t))_{1\leq i\leq n, t\geq 0}=
(\widetilde{M}_{i,K}(t\wedge\tilde{\theta}_{K,K'}), 
\tilde{r}_{i,K}(t\wedge\tilde{\theta}_{K,K'}))_{1\leq i\leq n, t\geq 0}.
\end{displaymath}
By coupling this way the processes for all values of $K$, one gets a process
$(\widetilde{M}_{i}(t), \tilde{r}_{i}(t))_{1\leq i\leq n, t\geq 0}$ such that
\begin{displaymath}
(\widetilde{M}_{i,K}(t), \tilde{r}_{i,K}(t))_{1\leq i\leq n, t\geq 0}=
(\widetilde{M}_{i}(t\wedge\tilde{\theta}_{K}), 
\tilde{r}_{i}(t\wedge\tilde{\theta}_{K}))_{1\leq i\leq n, t\geq 0},
\end{displaymath}
where $\tilde{\theta}_{K}$ is the first hitting time of $K$ or $-K$.
$\widetilde{M}_{i}$ is a martingale with respect to the filtration generated by
$(\widetilde{M}_{i}(t), \tilde{r}_{i}(t))_{1\leq i\leq n, t\geq 0}$ and
it satisfies 
\begin{displaymath}
\langle \widetilde{M}_{i},\widetilde{M}_{i}\rangle_{t}=\tilde{r}_{i}(t),
\qquad
\langle \widetilde{M}_{i},\widetilde{M}_{j}\rangle_{t}=0,~\text{for}~i\neq j.
\end{displaymath}

According to Dubins-Schwarz-Knight's theorem (see Theorem 1.9, chapter V, §1 in \cite{RevuzYor1999BMGrundlehren}), there are 
$n$ independent Brownian motions $(W^{(i)})_{1\leq i\leq n}$ and independent from
$(\tilde{r}_{i}(t))_{1\leq i\leq n, t\geq 0}$, such that
\begin{displaymath}
\widetilde{M}_{i}(t)=W^{(i)}(\tilde{r}_{i}(t)).
\end{displaymath}
With the notation of \eqref{EqDensityConduc}, the process 
$(X_{i}(t\wedge\theta_{K}), r_{i}(t\wedge\theta_{K}))_{1\leq i\leq n, t\geq 0}$
has a density
\begin{displaymath}
D(R(e_{1})-\varepsilon,\dots,R(e_{n})-\varepsilon)
\end{displaymath}
with respect the process
$(\widetilde{M}_{i}(t\wedge\tilde{\theta}_{K}), 
\tilde{r}_{i}(t\wedge\tilde{\theta}_{K}))_{1\leq i\leq n, t\geq 0}$.
Thus 
$(X_{i}(t), r_{i}(t))_{1\leq i\leq n, t\geq 0}$ has the same density
$D(R(e_{1})-\varepsilon,\dots,R(e_{n})-\varepsilon)$ with respect
$(\widetilde{M}_{i}(t),\tilde{r}_{i}(t))_{1\leq i\leq n, t\geq 0}$.
From Lemma \ref{LemAbsCont} follows that the process
$(\hat{\phi}_{z_{i}(r)})_{1\leq i\leq n, 0\leq r\leq R(e_{i})-\varepsilon}$ 
is distributed like the restriction of a metric graph GFF to 
\begin{displaymath}
\lbrace z_{i}(r)\vert 1\leq i\leq n, 0\leq r\leq R(e_{i})-\varepsilon\rbrace.
\end{displaymath}
Together with the Markov property follows that $\hat{\phi}$ is distributed like a metric graph GFF.

Finally, let us not assume anymore that $r_{i}(T)\leq R(e_{i})-\varepsilon$ for all $i$. Let $\tau^{\varepsilon}_{1}$
be the first time on $[0,T]$ when one of the $r_{i}(t)$ hits $R(e_{i})-\varepsilon$.
Let $\tau_{1}$ be the first time on $[0,T]$ when one $r_{i}(t)$ hits $R(e_{i})$.
Then the process 
$(X_{i}(t\wedge \tau^{\varepsilon}_{1}),r_{i}(t\wedge \tau^{\varepsilon}_{1}))_{1\leq i\leq n, 0\leq t\leq T}$
satisfies the lemma. By convergence
$(X_{i}(t\wedge \tau_{1}),r_{i}(t\wedge \tau_{1}))_{1\leq i\leq n, 0\leq t\leq T}$
satisfies the lemma, too.
Then one can consider successive stopping times $\tau_{1},\dots,\tau_{k}$ when one
of the $r_{i}(t)$ hits $R(e_{i})$. After each stopping time $\tau_{q}$, there is at least one process 
$(X_{i}(t),r_{i}(t))$ that is frozen, and applying our previous method to a smaller number of evolving processes, one can show that 
$(X_{i}(t\wedge \tau_{q}),r_{i}(t\wedge \tau_{q}))_{1\leq i\leq n, 0\leq t\leq T}$
satisfies the lemma for all $q\in \lbrace 1,\dots,k\rbrace$.
\end{proof}

\subsection {A particular exploration}
We now define a particular exploration of the GFF, keeping the same notation as before. Let us first give the formal definition:  
Given $r\in[0,R(e_{i})]$, we will denote by $L_{i}(r)$ the local time at zero accumulated by $\tilde{\phi}$ on the interval
$[x_{i},z_{i}(r)]$. We will construct an exploration $(r_{i}(t))_{1\leq i\leq n, t\geq 0}$ of $\tilde{\phi}$, such that for all $t\geq 0$,
\begin{equation}
\label{EqSameLT0}
L_{1}(r_{1}(t))=\dots=L_{n}(r_{n}(t)).
\end{equation}
{This is reminiscent of first-passage percolation, where one also looks at the growth of geodesic balls in a random metric. Note however that in the present setting,
the random lengths of edges are neither uniform nor independent, and that some edges have zero length.}

Denote
\begin{displaymath}
\ell_{max}:=\min(L_{1}(R(e_{1})),\dots,L_{n}(R(e_{n}))).
\end{displaymath}
For $\ell\in(0,\ell_{max}]$, define
\begin{displaymath}
R^{-}_{i}(\ell):=\sup\lbrace r\in[0,R(e_{i})]\vert L_{i}(r)<\ell\rbrace.
\end{displaymath}
and
\begin{displaymath}
R^{-}(\ell):=\sum_{i=1}^{n}R^{-}_{i}(\ell).
\end{displaymath}
Denote, if $h(x_{i})\neq 0$,
\begin{displaymath}
\rho^{0}_{i}:=\sup\lbrace r\in[0,R(e_{i})]\vert \tilde{\phi}_{z_{i}(\cdot)}
~\text{stays away from}~0~\text{on}~[0,r]\rbrace
.
\end{displaymath}
If $h(x_{i})= 0$, set $\rho^{0}_{i}=0$.
For $r\in (0,R^{-}_{i}(\ell_{max})]$, define
\begin{displaymath}
T_{i}(r)=
\left\lbrace 
\begin{array}{ll}
r+\sum_{j=1}^{i-1}\rho^{0}_{j} & \text{if}~r\leq\rho^{0}_{i}, \\ 
r-R^{-}_{i}(L_{i}(r)) + R^{-}(L_{i}(r)) & \text{if}~r>\rho^{0}_{i}.
\end{array} 
\right. 
\end{displaymath}

Define $r_{i}(t)$ as the inverse of $T_{i}$. If $t\in[0,R^{-}(\ell_{max})]$, let
\begin{displaymath}
r_{i}(t)=\inf\lbrace r\in(0,R^{-}_{i}(\ell_{max})]\vert T_{i}(r)\geq t\rbrace
\end{displaymath}
For $t>R^{-}(\ell_{max})$, we set $r_{i}(t)=R^{-}_{i}(\ell_{max})$.

Let us now explain in words what this processes $(r_{i}(t))_{1\leq i\leq n, t\geq 0}$ does. Each $r_{i}(t)$ starts from $0$.
First, $r_{1}(t)$ starts by exploring the interval $[0,\rho^{0}_{1}]$ while 
$r_{2}(t),\dots,r_{n}(t)$ stay at $0$. Then $r_{2}(t)$ explores $[0,\rho^{0}_{2}]$, etc.
At time $\sum_{i=1}^{n}\rho^{0}_{i}$, each $r_{i}(t)$ is at $\rho^{0}_{i}$. After that time, each
$r_{i}(t)$ explores the excursions of $\tilde{\phi}_{z_{i}(\cdot)}$ away from zero. The condition
\eqref{EqSameLT0} actually imposes the order in which the excursions across different values of $i$ are explored. At time $R^{-}(\ell_{max})$, each $r_{i}(t)$ reaches $R^{-}_{i}(\ell_{max})$ and does not
evolve after that time. For exactly one $i$, $R^{-}_{i}(\ell_{max})$ equals $R(e_{i})$ and $r_{i}(t)$
reaches the end of the interval corresponding to the edge $e_{i}$. Moreover,
for all $t\in [0,R^{-}(\ell_{max})]$, we have
$\sum_{i=1}^{n}r_{i}(t)=t$.

The folowing lemma sums up some easy properties of this process: 

\begin{lemma}
\label{LemExplor}
For $i\in \lbrace 1,\dots,n\rbrace$, the process $r_{i}(t)$ is non-decreasing continuous with values in $[0,R(e_{i})]$.
$r_{i}(0)=0$. For $t$ large enough, $r_{i}(t)=R^{-}_{i}(\ell_{max})$ and 
$\tilde{\phi}_{z_{i}(r_{i}(t))}=\tilde{\phi}_{z_{i}(R^{-}_{i}(\ell_{max}))}=0$.
Almost surely, there exactly one $i\in \lbrace 1,\dots,n\rbrace$, such that $R^{-}_{i}(\ell_{max})=R(e_{i})$.

The process $(r_{i}(t))_{1\leq i\leq n, t\geq 0}$ is an exploration of the metric graph GFF $\tilde{\phi}$ such that 
for all $t\geq 0$,
\begin{equation}
\label{EqSameLT}
L_{1}(r_{1}(t))=\dots=L_{n}(r_{n}(t)).
\end{equation}

Almost surely, for all $t\geq 0$, there is at most one $i\in \lbrace 1,\dots,n\rbrace$ such that
$r_{i}(t)>0$ and $\tilde{\phi}_{z_{i}(r_{i}(t))}\neq 0$. In particular, if $h$ is non-negative on $A$, then for any
$i,j\in\lbrace 1,\dots,n\rbrace$, $\tilde{\phi}_{z_{i}(r_{i}(t))}\tilde{\phi}_{z_{j}(r_{j}(t))}\geq 0$.
\end{lemma}

\begin{proof}
By definition, $r_{i}(t)\in[0,R^{-}_{i}(\ell_{max})]\subseteq [0,R(e_{i})]$, and for $t>R^{-}(\ell_{max})$, $r_{i}(t)=R^{-}_{i}(\ell_{max})$.
Also, by definition, $r_{i}(t)$ is non-decreasing. To show that it is continuous, we need to check that $T_{i}(r)$ is strictly increasing and that
$r_{i}(R^{-}(\ell_{max}))=R^{-}_{i}(\ell_{max})$. $T_{i}(r)$ is clearly strictly increasing on $(0,\rho_{i}^{0}]$. For $r>\rho_{i}^{0}$,
\begin{displaymath}
T_{i}(r)=r+\sum_{j\neq i}R^{-}_{j}(L_{i}(r)).
\end{displaymath}
The first term $r$ is strictly increasing and the others are non-decreasing. Thus $T_{i}(r)$ is strictly increasing on $(\rho_{i}^{0},R^{-}(\ell_{max})]$.
Moreover, $R^{-}_{j}(L_{j}(r))\geq\rho^{0}_{j}$. Thus, for $r>\rho_{i}^{0}$,
\begin{displaymath}
T_{i}(r)>\sum_{j=1}^{n}\rho^{0}_{j}=T_{i}(\rho_{i}^{0}).
\end{displaymath}
So $T_{i}(r)$ is strictly increasing on $(0,R^{-}_{i}(\ell_{max})]$ and consequently $r_{i}(t)$ is continuous on $[0,R^{-}_{i}(\ell_{max})]$ and 
$r_{i}(0)=0$. Moreover,
\begin{displaymath}
T_{i}(R_{i}^{-}(\ell_{max}))=R^{-}(L_{i}(R^{-}_{i}(\ell_{max})))=R^{-}(\ell_{max}).
\end{displaymath}
Thus $r_{i}(R^{-}(\ell_{max}))=R_{i}^{-}(\ell_{max})$ and $r_{i}(t)$ is continuous on $[0,+\infty)$.

Let us check now \eqref{EqSameLT}. If $t\leq \sum_{i=1}^{n}\rho_{i}^{0}$, then $r_{i}(t)\leq \rho_{i}^{0}$ and
\begin{displaymath}
L_{1}(r_{1}(t))=\dots=L_{n}(r_{n}(t))=0.
\end{displaymath}
If $t\geq R^{-}(\ell_{max})$, then $r_{i}(t)=R^{-}_{i}(\ell_{max})$ and
\begin{displaymath}
L_{1}(r_{1}(t))=\dots=L_{n}(r_{n}(t))=\ell_{max}.
\end{displaymath}
Consider now that
\begin{displaymath}
\sum_{i=1}^{n}\rho_{i}^{0}<t<R^{-}(\ell_{max}).
\end{displaymath}
Assume that there is $i\neq j$ such that $L_{i}(r_{i}(t))<L_{j}(r_{j}(t))$.
Let $\ell\in (L_{i}(r_{i}(t)),L_{j}(r_{j}(t)))$. Then $R^{-}_{j}(\ell)<r_{j}(t)$.
Moreover,
\begin{displaymath}
t\leq T_{i}(r_{i}(t)) < R^{-}(\ell)=T_{j}(R^{-}_{j}(\ell)).
\end{displaymath}
By definition of $r_{j}$, this means that $r_{j}(t)\leq R^{-}_{j}(\ell)$, which contradicts $R^{-}_{j}(\ell)<r_{j}(t)$.
This means that we cannot have $L_{i}(r_{i}(t))<L_{j}(r_{j}(t))$.

Let us check now that $(r_{i}(t))_{1\leq i\leq n, t\geq 0}$ is an exploration of $\tilde{\phi}$.
Let $t>0$. Let $(r_{1},\dots,r_{n})\in\prod_{i=1}^{n}[0,R(e_{i})]$. Define
\begin{displaymath}
\underline{\ell}(r_{1},\dots,r_{n}):=\min(L_{1}(r_{1}),\dots,L_{n}(r_{n})).
\end{displaymath}
Let $\underline{i}$ be the smallest $i$ such that $r_{i}<\rho_{i}^{0}$ if such exists, $n$ otherwise. The random variables
$\underline{\ell}(r_{1},\dots,r_{n})$ and $\underline{i}$ are $\mathcal{F}(r_{1},\dots,r_{n})$-measurable
If $\underline{\ell}(r_{1},\dots,r_{n})=0$, then 
\begin{equation}
\label{EqCondExplor}
r_{i}(t)\leq r_{i}~\text{for all}~i\in \lbrace 1,\dots,n\rbrace
\end{equation}
is equivalent to
\begin{displaymath}
t\leq r_{\underline{i}}\wedge \rho_{\underline{i}}^{0}.
\end{displaymath}
If $\underline{\ell}(r_{1},\dots,r_{n})>0$, then \eqref{EqCondExplor} is equivalent to
\begin{displaymath}
t\leq R^{-}(\underline{\ell}(r_{1},\dots,r_{n}))+
\sum_{\substack{1\leq i\leq n\\L_{i}(r_{i})=\underline{\ell}(r_{1},\dots,r_{n})}}
(r_{i}-R_{i}^{-}(\underline{\ell}(r_{1},\dots,r_{n}))).
\end{displaymath}
Thus the event \eqref{EqCondExplor} is $\mathcal{F}(r_{1},\dots,r_{n})$-measurable.

If $\tilde{\phi}_{z_{i}(r_{i}(t))}\neq 0$, then either $r_{i}(t)<\rho^{0}_{i}$ or
$r_{i}(t)>\rho^{0}_{i}$ and $R^{-}_{i}(\ell)$ has a discontinuity at $\ell=L_{i}(r_{i}(t))$. If 
$r_{i}(t)<\rho^{0}_{i}$, then by definition, for $j\neq i$, either $r_{j}(t)=0$ or
$r_{j}(t)=\rho^{0}_{j}$ and $\tilde{\phi}_{z_{i}(r_{j}(t))}=\tilde{\phi}_{z_{i}(\rho^{0}_{j})}=0$.
Consider now the case when $r_{i}(t)>\rho^{0}_{i}$ and $R^{-}_{i}(\ell)$ has a discontinuity at 
$\ell=L_{i}(r_{i}(t))$.
To show that for $j\neq i$, $\tilde{\phi}_{z_{i}(r_{j}(t))}=0$, we need to show that 
$R^{-}_{j}(\ell)$ is continuous at $\ell=L_{i}(r_{i}(t))=L_{j}(r_{j}(t))$. So we need to check that
for $i\neq j\in \lbrace 1,\dots,n\rbrace$, almost surely, the discontinuity sets of $R^{-}_{i}(\ell)$ and $R^{-}_{j}(\ell)$ are
disjoint. This comes from the fact that the law of $(\tilde{\phi}_{z_{i}(\cdot)},\tilde{\phi}_{z_{i}(\cdot)})$
is absolutely continuous with respect the law of two independent Brownian motions and the fact that for two independent Brownian motions the discontinuity
sets of the inverse of the local time at zero (corresponding to the local time levels of Brownian excursions away from zero) are disjoint.
\end{proof}

Next is the key lemma for the proof of Proposition \ref{thmGeneralisedLevy}.

\begin{lemma}
\label{LemGenLevySDE}
Assume that the boundary condition $h$ is non-negative on $A$. Let $I_{max}$ denote the maximum over $1\leq i\leq n$ of the quantities $ I_{y_{i},A}(\tilde{\phi})$ 
and let $\rho_i$ denote the first $r$ at which $I_{z_{i}(r),A}$ reaches $I_{max}$. 
Then, the processes
\begin{displaymath}
(\vert\tilde{\phi}_{z_{i}(r)}\vert, \delta_{z_{i}(r),A})
_{1\leq i\leq n, 0\leq r\leq R_{i}^{-}(\ell_{max})}
\end{displaymath}
and
\begin{displaymath}
(\tilde{\phi}_{z_{i}(r)}-I_{z_{i}(r),A}, -I_{z_{i}(r),A})
_{1\leq i\leq n,  0 \le r \le \rho_i}
\end{displaymath}
have the same distribution.
\end{lemma}

\begin{proof}
Let $(r_{i}(t))_{1\leq i\leq n, t\geq 0}$ be the exploration of the metric graph GFF $\tilde{\phi}$
that appears in Lemma \ref{LemExplor}. According to Tanaka's formula 
(see Theorem 1.2, chapter VI, §1 in \cite{RevuzYor1999BMGrundlehren}), we have the following semi-martingale decomposition for $\vert\tilde{\phi}_{z_{i}(r_{i}(t))}\vert$:
\begin{displaymath}
\vert\tilde{\phi}_{z_{i}(r_{i}(t))}\vert=
h(x_{i})+\int_{0}^{t}\hbox {sgn}(\tilde{\phi}_{z_{i}(r_{i}(s))}) 
d\tilde{\phi}_{z_{i}(r_{i}(s))}+L_{i}(r_{i}(t)).
\end{displaymath}
With the notation of Lemma \ref{LemSDE},
\begin{eqnarray*}
\lefteqn { 
\vert\tilde{\phi}_{i}(t)\vert-L_{i}(r_{i}(t))} \\
&=&
h(x_{i})+\int_{0}^{t}\hbox {sgn}(\tilde{\phi}_{i}(s)) 
d\tilde{\phi}_{i}(s)
\\&=&h(x_{i})+\int_{0}^{t}\hbox {sgn}(\tilde{\phi}_{i}(s))dM_{i}(s)
+\sum_{j\neq i}\int_{0}^{t}\hbox {sgn}(\tilde{\phi}_{i}(s))
C^{\rm eff}_{ij}(s)(\tilde{\phi}_{j}(s)-\tilde{\phi}_{i}(s)) dr_{i}(s).
\end{eqnarray*}
By construction, $L_{i}(r_{i}(s))=L_{j}(r_{j}(s))$, and according to Lemma \ref{LemExplor}, 
$\tilde{\phi}_{i}(s)\tilde{\phi}_{j}(s)\geq 0$.
Thus,
\begin{eqnarray*}
\lefteqn { \vert\tilde{\phi}_{i}(t)\vert-L_{i}(r_{i}(t))}\\
&=& h(x_{i})+\int_{0}^{t}\hbox {sgn}(\tilde{\phi}_{i}(s))dM_{i}(s)
+\sum_{j\neq i}\int_{0}^{t}
C^{\rm eff}_{ij}(s)(\vert\tilde{\phi}_{j}(s)\vert-L_{j}(r_{j}(s))
-\vert\tilde{\phi}_{i}(s)\vert+L_{i}(r_{i}(s))) dr_{i}(s).
\end{eqnarray*}
Thus, the processes $(\vert\tilde{\phi}_{i}(t)\vert-L_{i}(r_{i}(t)))_{1\leq i\leq n, t\geq 0}$
satisfy  \eqref{EqSDEX} and \eqref{EqSDEXQV}. Let $T>0$. According to Lemma
\ref{LemSDE2}, $(\vert\tilde{\phi}_{i}(t)\vert-L_{i}(r_{i}(t)))_{1\leq i\leq n, 0\leq t\leq T}$ can be extended to a GFF $\hat{\phi}^{T}$ on the whole metric graph $\widetilde{\mathcal{G}}$. By letting $T$ tend to infinity, we get that $(\vert\tilde{\phi}_{i}(t)\vert-L_{i}(r_{i}(t)))_{1\leq i\leq n, t\geq 0}$
extends to a GFF $\hat{\phi}$ on whole $\widetilde{\mathcal{G}}$.
For that GFF we have
\begin{displaymath}
L_{i}(r_{i}(t))=-I_{z_{i}(r_{i}(t)),A}(\hat{\phi}).
\end{displaymath}
Moreover, $L_{i}(r_{i}(t))=\delta_{z_{i}(r_{i}(t)),A}(\tilde{\phi})$. 
This implies our lemma.
\end{proof}

The idea is now to use the previous lemma to prove Proposition \ref {thmGeneralisedLevy} by induction on the number of edges of the graph. 
We first need to check that the proposition holds in the case where the graph consists only of one edge, which is the case of the Brownian bridge described at the beginning of the introduction: 
\begin{lemma}[Bertoin-Pitman, \cite {BertoinPitman}]
\label{LemLevyBridge}
Proposition \ref {thmGeneralisedLevy} holds when the graph $\mathcal{G}$ is just two vertices $x_{1}$ and $x_{2}$ with only one edge $e$
joining them.
\end{lemma}
As already mentioned in the introduction, this result is due to Bertoin and Pitman (Theorem 4.1 in \cite{BertoinPitman}) and there are several possible ways to derive it. 
{In \cite{BertoinPitman}, it was obtained via a combination of several classical transforms (such as the Verwaat transform) between the Brownian bridge, the Brownian excursion and the Brownian meander.}
The following short proof {differs from} that of \cite {BertoinPitman}, and is quite natural from our general metric graph perspective. 

\begin{proof}
For $r\in [0,R(e)]$, denote $z_{1}(r)$ the point at distance $r$ from $x_{1}$ and $z_{2}(r)$ the point
at distance $r$ from $x_{2}$. $L_{1}(r)$ respectively $L_{2}(r)$ will denote the local time at zero accumulated by $\tilde{\phi}$ on $[x_{1},z_{1}(r)]$, respectively on $[z_{2}(r),x_{2}]$.

As described previously (see Lemma \ref{LemExplor}), one can construct the pair of continuous non-decreasing processes
$(r_{1}(t),r_{2}(t))_{t\geq 0}$ such that
\begin{itemize}
\item $r_{1}(0)=r_{2}(0)=0$,
\item $(r_{1}(t),r_{2}(t))_{t\geq 0}$ is an exploration of $\tilde{\phi}$,
that is to say the event $\lbrace r_{1}(t)\leq r_{1}, r_{2}(t)\leq r_{2}\rbrace$ is measurable
with respect the restriction of $\tilde{\phi}$ to
$[x_{1},z_{1}(r)]\cup [z_{2}(r),x_{2}]$,
\item for all $t\geq 0$, $L_{1}(r_{1}(t))=L_{2}(r_{2}(t))$,
\item for all $t\geq 0$, $z_{1}(r_{1}(t))\leq z_{2}(r_{2}(t))$,
that is to say $r_{1}(t)+r_{2}(t)\leq R(e)$,
\item for $t$ large enough $z_{1}(r_{1}(t))=z_{2}(r_{2}(t))$ and
$(r_{1}(t),r_{2}(t))$ does not evolve any more,
\item for all $t\geq 0$, 
$\tilde{\phi}_{z_{1}(r_{1}(t))}\tilde{\phi}_{z_{2}(r_{2}(t))}\geq 0$.
\end{itemize}
By construction,
$L_{i}(r_{i}(t))=\delta_{z_{i}(r_{i}(t)),A}$, $i=1,2$.

The process $(\tilde{\phi}_{z_{1}(r_{1}(t))},\tilde{\phi}_{z_{2}(r_{2}(t))})_{t\geq 0}$
satisfies the SDE \eqref{EqSDEExplor},\eqref{EqVarQuad}, where
\begin{displaymath}
C^{\rm eff}_{12}(s)=(R(e)-r_{1}(s)-r_{2}(s))^{-1}.
\end{displaymath}

For $\varepsilon\in (0,R(e))$, $\tau^{\varepsilon}$ will denote the stopping time when
$r_{1}(t)+r_{2}(t)$ hits $R(e)-\varepsilon$. The stopped process
$(\vert\tilde{\phi}_{z_{1}(r_{1}(t\wedge\tau^{\varepsilon}))}\vert
-L_{1}(r_{1}(t\wedge\tau^{\varepsilon})),
\vert\tilde{\phi}_{z_{2}(r_{2}(t\wedge\tau^{\varepsilon}))}\vert
-L_{2}(r_{2}(t\wedge\tau^{\varepsilon})))_{t\geq 0}$ satisfies the SDE
\eqref{EqSDEX},\eqref{EqSDEXQV} and one can apply Lemma
\ref{LemSDE2}, to show that by interpolating the stopped process above
by an independent Brownian bridge of length $\varepsilon$ from
$-L_{1}(r_{1}(\tau^{\varepsilon}))$ to
$-L_{2}(r_{2}(\tau^{\varepsilon}))=-L_{1}(r_{1}(\tau^{\varepsilon}))$, one gets
a Brownian bridge of length $R(e)$ from $h(x_{1})$ to $h(x_{2})$. By letting
$\varepsilon$ tend to zero, we get that 
$(\vert\tilde{\phi}_{x}\vert-\delta_{x,A})_{x\in [x_{1},x_{2}]}$ 
is distributed like $\tilde{\phi}$, which is equivalent to the statement of the lemma.
\end{proof}

We are now finally able to prove Proposition \ref {thmGeneralisedLevy}:

\begin{proof}
We will prove the proposition by induction on the number of edges $\vert E\vert$ in the graph. The case where this number is equal to one is covered by Lemma \ref {LemLevyBridge}.

Suppose now that $|E| \ge 2$, and that Proposition \ref {thmGeneralisedLevy} holds for all graphs with fewer edges. 
As explained before, the edges that do join points of $A$ together can be removed (mind that one uses
Lemma \ref {LemLevyBridge} to justify this), so that if there is such an edge in the graph, then the induction hypothesis allows to conclude that the result holds also for the graph $(V,E)$. 

We now assume that no edge of $E$ joins two sites of $A$. 
Using the notation of the beginning of Section \ref{SecLevy} and Lemma
\ref{LemGenLevySDE}, we get that the processes
\begin{displaymath}
(\vert\tilde{\phi}_{z_{i}(r)}\vert - \delta_{z_{i}(r),A})
_{1\leq i\leq n, 0\leq r\leq R_{i}^{-}(\ell_{max})}
\hbox { and } (\tilde{\phi}_{z_{i}(r)})_{1\leq i\leq n, r \le \rho_i}
\end {displaymath} 
where
$I_{max}=\max_{1\leq i\leq n} I_{y_{i},A}$,
have the same distribution. The sets
\begin{equation}
\label{EqExploredSet1}
\left\lbrace z_{i}(r)\Big\vert 1\leq i\leq n, 0\leq r\leq R_{i}^{-}(\ell_{max})\right\rbrace
\end{equation}
and
\begin{equation}
\label{EqExploredSet2}
\left\lbrace z_{i}(r)\Big\vert 1\leq i\leq n, I_{z_{i}(r),A}\geq I_{max}\right\rbrace
\end{equation}
are each made of $n$ lines contained inside $[x_{i},y_{i}]$, $1\leq i\leq n$. 
Moreover, in each case, for exactly one $i$, the
interval $[x_{i},y_{i}]$ is entirely contained in the corresponding set. 
Conditionally on 
$(\tilde{\phi}_{z_{i}(r)})_{1\leq i\leq n, 0\leq r\leq R_{i}^{-}(\ell_{max})}$,
outside the set \eqref{EqExploredSet1}, $\tilde{\phi}$ is distributed like a metric graph GFF
with boundary condition 0 at $z_{i}(R_{i}^{-}(\ell_{max}))$, $1\leq i\leq n$.
Conditionally on
$(\tilde{\phi}_{z_{i}(r)})_{1\leq i\leq n, r < \rho_i }$,
outside the set \eqref{EqExploredSet2}, $\tilde{\phi}$ is distributed like a metric graph GFF
with boundary condition $I_{max}$. Since we already have Lemma \ref{LemGenLevySDE}, to show the proposition, we need to check that it is satisfied by a metric graph GFF on $\widetilde{\mathcal{G}}$ minus 
\eqref{EqExploredSet1} with boundary condition zero. But the metric graph $\widetilde{\mathcal{G}}$ minus 
\eqref{EqExploredSet1} has one less edge, thus we can use the induction hypothesis. This concludes our proof.
\end{proof}

\section {Invariance under rewiring}

\label{SecExplicit}
\label{S3}

\subsection {Main statement and some comments}

Let us first state the more general version of Proposition~\ref{PropLaw2Points}. 
Consider a finite graph $\mathcal{G}=(V,E)$ as before, and  assume that the boundary set $A\subseteq V$ contains at least two vertices. We also  consider a partition
of $A$ into two non-empty subsets $\widehat{A}$ and $\widecheck{A}$ so  that the sign of the boundary condition $h$ on $\widehat A$ is constant, and that its sign on $\widecheck A$ is also constant (the value $0$ is allowed for some boundary points, and the 
signs on $\widehat A$ and on $\widecheck A$ do not need to be the same). 
As before, we will consider the effective conductance matrix $C^{\rm eff}_A$ of the circuit corresponding to the boundary $A$. 
We then consider the metric graph Gaussian free field $\tilde{\phi}$ with boundary condition $h$ on $A$, and we more specifically study the random variable
$$\delta_{\widehat{A},\widecheck{A}} = \delta_{\widehat{A},\widecheck{A}} ( \tilde \phi)  := \min_{\hat x \in \widehat A, \check x \in \widecheck A} \delta_{\hat x , \check x } $$
that measures the pseudo-distance between these two parts of the boundary. 

\begin{proposition}
\label{thmDist2Sets}
The law of the non-negative random variable
$\delta_{\widehat{A},\widecheck{A}}$ is described by the fact that for all $\ell > 0$, 
\begin{equation}
\label{EqExplDist}
\mathbb{P}(\delta_{\widehat{A},\widecheck{A}} \geq  \ell)=
\prod_{\substack{\hat{x}\in\widehat{A}\\\check{x}\in\widecheck{A}}}
\exp\bigg(-\frac{1}{2}C^{\rm eff}_A (\hat{x},\check{x})
\big(\vert h(\hat{x})\vert +\vert h(\check{x})\vert+\ell\big)^{2}
+\frac{1}{2}C^{\rm eff}_A (\hat{x},\check{x})\big(h(\hat{x})-h(\check{x})\big)^{2}\bigg).
\end{equation}
\end {proposition}
We can also note that if the sign of $h$ is the same on $\widecheck A$ as on $\widehat A$, then this implies that
\begin{equation}
\label{EqProbConnect}
\mathbb{P}(\delta_{\widehat{A},\widecheck{A}}> 0)=
\prod_{\substack{\hat{x}\in\widehat{A}\\\check{x}\in\widecheck{A}}}
\exp\left(-2C^{\rm eff}_A (\hat{x},\check{x}) h(\hat{x})h(\check{x}) \right).
\end{equation}

Note that in the special case where $\widehat A$ and $\widecheck A$ are both singletons (or equivalently, when $h$ is constant on both $\hat A$ and $\check A$), this proposition is  exactly Proposition~\ref{PropLaw2Points}. Also, when $h$ is the zero function, then $(\delta_{\widehat A, \widecheck A})^2$ is exponentially distributed.

Note also that Lemma \ref {LemLevyBridge} implies the proposition in the special case where the graph consists only of edges that join directly points from $\widecheck A$ to points from $\widehat A$. Indeed, in this case, $\delta_{\widehat{A},\widecheck{A}} >  \ell$ if and only if the length of each of the edges in greater than $\ell$, and we can note that the events corresponding to these edges  are all independent (and their laws are given by Lemma \ref {LemLevyBridge}). Hence, given what we have established so far, proving Proposition \ref {thmDist2Sets} is equivalent to proving that: 

\begin {proposition}
\label {reformulation}
 Suppose that two finite graphs ${\mathcal G}$ and ${\mathcal G}'$ share the same boundary set $A$, that $A$ is partitioned into two non-empty sets $\widehat A$ and $\widecheck A$. 
 Suppose further that for each $\hat x \in \widehat A$ and $\check x \in \widecheck A$, the effective resistance $R^{\rm eff}_A (\hat x, \check x)$ is the same for both graphs. Then, for each boundary condition $h$ with constant signs on each of $\widehat A$ and $\widecheck A$, the law of $\delta_{\widehat A, \widecheck A}$ is the same for both graphs. 
\end {proposition}

In other words, one can view this  as  an invariance of the distribution of $\delta_{\widehat A, \widecheck A}$ 
under rewiring of the electrical network that preserves the effective conductance matrix $C_A^{\rm eff}$.

Let us also note that  \eqref{EqProbConnect} can also be reformulated as follows: For general boundary conditions $h$ on $A$, if $a < \min_A h$, 
then the event $E_a$ that there exists a continuous path
$\gamma$ connecting  $\widehat{A}$ to $\widecheck{A}$ along which $\tilde \phi$ remains greater than $a$, 
is exactly the event $\{ \delta_{\widehat{A},\widecheck{A}}(\tilde{\phi}-a) = 0 \}$, so that  
\begin{displaymath}
\mathbb{P}(E_a )=
1- \prod_{\substack{\hat{x}\in\widehat{A}\\\check{x}\in\widecheck{A}}}
\exp\left(-2C^{\rm eff}_A (\hat{x},\check{x})(h(\hat{x})-a)(h(\check{x})-a)\right).
\end{displaymath}
This can then be used to establish a lower bound for the probability that 
$\widehat{A}$  and $\widecheck{A}$ are connected by a path on which the height of the discrete Gaussian free field is greater than some value.
For similar lower bounds, using different techniques, see \cite{Sznitman2015PercGFF}, Section 3, and \cite{DingLi2016ChimGFF}, Proposition 2.1.

  \medbreak
 
 Let us first make a few further comments.
 \begin {itemize}
 \item 
In all these statements,  edges of the graph that join two points of $\widehat A$ (or two points of $\widecheck A$) do play no role in the distance $\delta_{\widehat A, \widecheck A}$, and the GFF on those edges is independent from the rest for a given $h$. Hence, we do not need to bother about these edges. 
 \item 
In the special case where the graph is a three-legged ``star-graph'' ${\mathcal G}_3$ with $V= \{ \hat x_1, \hat x_2, y, \check x \}$ and three edges joining respectively $y$ to the three boundary point $\hat x_1$, $\hat x_2$ and $\check x$, then for $\widehat A = \{ \hat x_1,  \hat x_2 \}$, it is easy to check directly that Proposition \ref{thmDist2Sets} holds: Indeed,
$$\delta_{\widehat A, \check x} = \min ( \delta_{\hat x_1, y} , \delta_{\hat x_2, y}) + \delta_{y, \check x},$$ and one can first sample $\tilde \phi (y)$ and because these three random edge-lengths
are then conditionally independent,  one can conclude. 
\item
The previous case of the star-graph ${\mathcal G}_3$ shows that the law of $\delta_{\widehat A, \widecheck A}$ in this case is identical to that when the star-graph is replaced by the electrically equivalent 
triangle graph. 
\begin{figure}[ht!]
  \centering 
  \includegraphics[width=.5\textwidth]{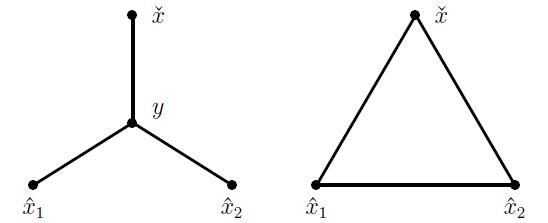}
  \caption{
{Star-triangle} transformation. 
}
  \label{figStarDelta}
\end{figure}
In fact, the previous observation can be immediately generalised to the case of star graphs ${\mathcal G}_N$ with more boundary points (when the graph consists only of edges joining one single central point $y$ to each boundary point in $\widehat A \cup \widecheck A$ but this time, the number $N$ of sites of $A$ can be greater than $3$). In that case, replacing this graph by the complete graph on $A$ with the same effective resistances will also not change the law of $\delta_{\widehat A, \widecheck A}$.
\end {itemize} 

Recall that the electrical network defined by ${\mathcal G}$ can always be ``reduced'' to the electrically equivalent network (viewed from $A$) consisting only of edges joining directly the edges of $A$, by a finite iteration of local ``star ${\mathcal G}_N$ to complete graph'' transformations (at each step, one chooses a site that is not in $A$ and removes it using this transformation). Proving Proposition \ref {reformulation} is therefore equivalent to showing that the law of $\delta_{\widehat A, \widecheck A}$ remains 
unchanged under each of such local transformation. One may wonder if one simple direct proof of this fact might just be to see that if one conditions 
on the values of the GFF outside of the considered star (which is the same before and after the local transformation), then the conditional law of $\delta_{\widehat A, \widecheck A}$ remains unchanged under this local transformation. In the coming few paragraphs, we now explain that this is {\em not} the case -- this will also provide some 
justification for the actual strategy of our proof. 

Let us first consider the case where the graph ${\mathcal G}$ is just a three-legged star with $A = \{ \hat x_1, \hat x_2, \check x\}$ and the inner point $y$ as before. This 
set is electrically equivalent to a triangle graph ${\mathcal G'}$ consisting only of three edges joining the points of $A$. It is easy to see that it is {\em not true} that for any non-negative boundary condition $h$, the joint law of the couple $(\delta_{\hat x_1, \check x}, \delta_{\hat x_2, \check x})$ is the same for both graphs (recall that the proposition only states than the minimum of these two quantities have the same law for both graphs). 
Indeed, one can for instance consider the case where $h$ is constant and equal to $a$ on $A$, and look at the asymptotic behaviour when $a \to 0$ of the probability 
that $(\delta_{\hat x_1, \check x}, \delta_{\hat x_2, \check x}) = (0, 0 )$ on each of the graphs. One can directly compute and compare these probabilities, but one can also 
proceed heuristically as follows:
For $(\delta_{\hat x_1, \check x}, \delta_{\hat x_2, \check x}) = (0, 0 )$ to hold on the triangle graph, we see that the GFF has to be positive on each of the three edges, and 
it is easy to check that when $a \to 0$, this occurs with a probability equivalent to a constant times $a^3$ (basically, each of the GFF's is starting from $a$ near the three 
boundary points and has to avoid the origin). On the other hand, one the triangle graph, $(\delta_{\hat x_1, \check x}, \delta_{\hat x_2, \check x}) = (0, 0 )$ if and only 
the GFF is positive on two of the three edges, which occurs with a probability equivalent to a constant times $a^4$ (this time, for each of these two edges, the GFF has to remain positive near both of the endpoints of the edge). 
Hence, we see that Proposition \ref {thmDist2Sets} can not be generalised to statements that involve the joint law of more than 
two more than one  distance between  parts of the boundary.

One could nevertheless wonder (and this would indicate that the proof of Proposition \ref {thmDist2Sets} via conditioning and local rewirings 
could work) whether the following more general general result could hold: Suppose that we are now given a non-negative function $u$ on $A$, and we define 
$$\delta^u_{\widehat A, \widecheck A} := \min_{\hat x \in \widehat A, \check x \in \widecheck A}  ( \delta_{\hat x, \check x} + u(\hat x) + u(\check x)).$$
Then, is it the case that for each given $u$ and $h$, the law of this quantity is invariant under the rewiring procedure described in Proposition \ref{reformulation}? 
Again, as we shall now explain, the answer to this question is negative already in the case of the three-legged star graph:
Suppose that for the three-legged graph ${\mathcal G}_{3}$ and its electrical equivalent triangle graph 
${\mathcal G}'_{3}$, with $\widehat A = \{\hat  x_1, \hat x_2 \}$ and each given $h$ and $u$, 
the laws of $\delta^u_{\widehat A, \widecheck A}$ are the same for both graphs. By definition, this would imply that for each given $h$, and each $\ell$, $u_1$ and $u_2$, 
the probabilities of the event $\{ \delta_{\hat x_1, \check x} > \ell - u_1, \delta_{\hat x_2, \check x} > \ell - u_2 \}$ are the same for both graphs. This would imply that 
the joint law of the couple $(\delta_{\hat x_1, \check x} , \delta_{\hat x_2, \check x} )$ are the same for both graphs, but we have just argued that this is not the case.

Consider now the three-legged graph ${\mathcal G}_3$, and divide each of the two edges $(\hat x_1, y)$ and $(\hat x_2, y)$ into two pieces (by introducing additional vertices $x_1$ and $x_2$). 
Then, what we have just explained is that for this new graph (with six vertices), the conditional distribution of $\delta_{ \widehat A, \widecheck A}$ given what happens 
on the edges $(\hat x_1, x_1)$ and $(\hat x_2, x_2)$ is {\em not} invariant under the transformation that turns the star $(x_1, x_2, y, \check y)$ into a triangle
(in that setting, the function $u$ can be understood in terms of the distances $\delta_{\hat x_1, x_1}$ and $\delta_{\hat x_2, x_2}$ that have already been discovered).

This  does therefore  show that the previously mentioned possible proof of Proposition \ref {reformulation} via conditioning and local modifications can not work. It also 
suggests that a good exploration procedure to use in order to establish Proposition \ref {thmDist2Sets} is to start from the boundary and to discover simultaneously all 
points that are at the same distance from the boundary. Indeed, in the remaining-to-be-discovered domain, no non-constant ``delay-type terms'' like $u$ will appear.

\subsection {Proof of Proposition \ref {thmDist2Sets}}
\label{SubSevIdMetricGraph}
We will use the similar exploration martingale approach as for Proposition \ref {thmGeneralisedLevy}.
As in Section \ref{SecLevy}, we will consider a particular exploration 
$\hat{r}_{i}(t)\in [0,R(\hat{e}_{i})]$ of the metric
graph GFF $\tilde{\phi}$, such that for all
$t\geq 0$,
\begin{displaymath}
L_{1}(\hat{r}_{1}(t))=\dots =L_{\hat{n}}(\hat{r}_{\hat{n}}(t)),
\end{displaymath}
where $L_{i}(\hat{r}_{i}(t))$ is the local time at zero accumulated on the line
$[\hat{x}_{i},\hat{z}_{i}(\hat{r}_{i}(t))]$. 
We will denote by $L(t)$ the common value of $L_{i}(\hat{r}_{i}(t))$.
As before, the exploration will  eventually stop evolving at time 
$\widehat{T}$ when one of the $\hat{r}_{i}(t)$ hits
$R(\hat{e}_{i})$ (a.s., the corresponding value of $i$ is unique, see Lemma \ref{LemExplor}). In particular, if for all
$i,j\in \lbrace 1,\dots,\hat{n}\rbrace$, $h(\hat{x}_{i})h(\hat{x}_{j})\geq 0$, then
for all $i,j\in \lbrace 1,\dots,\hat{n}\rbrace$ and all $t\geq 0$, 
\begin{displaymath}
\tilde{\phi}_{\hat{z}_{i}(\hat{r}_{i}(t))}
\tilde{\phi}_{\hat{z}_{j}(\hat{r}_{j}(t))}\geq 0.
\end{displaymath}

Consider a time $t\in (0, \widehat{T})$. If we replace the circuit outside edges that join two vertices in $\widehat{A}$ or two vertices in $\widecheck{A}$, and outside the lines 
$[\hat{x}_{i},\hat{z}_{i}(\hat{r}_{i}(t))]$, $1\leq i\leq \hat{n}$, by a complete graph
joining $\lbrace \hat{z}_{1}(\hat{r}_{1}(t)), \ldots , \hat{z}_{n}(\hat{r}_{n}(t))\rbrace\cup\widecheck{A}$, and have an electrically equivalent circuit, then we will denote the obtained conductances by
\begin{itemize}
\item $C^{\rm eff}_{i j}(t)$ for the conductance of the edge between $\hat{z}_{i}(\hat{r}_{i}(t))$ and
$\hat{z}_{j}(\hat{r}_{i}(t))$,
\item $C^{\rm eff}_{i\check{x}}(t)$
for the conductance of the edge between $\hat{z}_{i}(\hat{r}_{i}(t))$ and $\check{x}\in \widecheck{A}$,
\item $C^{\rm eff}_{\check{x}\check{x}'}(t)$ for the conductance of the edge between
$\check{x}\neq\check{x}'\in \widecheck{A}$.
\end{itemize}
All these conductances have limits at $t=0$ and
\begin{displaymath}
\lim_{t\rightarrow 0}C^{\rm eff}_{ij}(t)=
C^{\rm eff}(\hat{x}_{i},\hat{x}_{j})~\text{if}~\hat{x}_{i}\neq\hat{x}_{j},
\qquad
\sum_{\substack{1\leq i\leq\hat{n}\\\hat{x}_{i}=\hat{x}}}
\lim_{t\rightarrow 0}C^{\rm eff}_{i\check{x}}(t)=C^{\rm eff}(\hat{x},\check{x}),
\end{displaymath}
\begin{displaymath}
\lim_{t\rightarrow 0}C^{\rm eff}_{\check{x}\check{x}'}(t)=C^{\rm eff}(\check{x},\check{x}').
\end{displaymath}
If $i_{0}$ is the value such that $\hat{r}_{i_{0}}(\widehat{T})=R(\hat{e}_{i_{0}})$ and if
$\hat{y}_{i_{0}}\in \widecheck{A}$, then
\begin{displaymath}
\lim_{t\rightarrow \widehat{T}} C^{\rm eff}_{i_{0}\hat{y}_{i_{0}}}(t) = +\infty.
\end{displaymath}
In all other cases the limits of effective conductances at $\widehat{T}$ are finite. 

Let $\ell > 0$. Let $\tau_{\ell}\in [0,+\infty]$ be the first time 
$L(t)$ hits the level $\ell$.
Consider the process $\Psi_{\ell}(t)$, defined for
$t\in [0, \widehat{T})$ as
\begin{eqnarray}
\label{EqDefPsiell}
  \Psi_{\ell}(t) &:= &
\prod_{\substack{ 1\leq i\leq\hat{n}\\\check{x}\in\widecheck{A}}}
\exp\bigg(
-\frac{1}{2}C^{\rm eff}_{i\check{x}}(t\wedge \tau_{\ell})
\big(\vert\tilde{\phi}_{\hat{z}_{i}(\hat{r}_{i}(t\wedge \tau_{\ell}))}\vert +
\vert h(\check{x})\vert+
\ell-L(t\wedge\tau_{\ell})\big)^{2} 
\\
\nonumber
&& \null  \quad +\frac{1}{2}C^{\rm eff}_{i\check{x}}(t\wedge \tau_{\ell})
\big(\tilde{\phi}_{\hat{z}_{i}(\hat{r}_{i}(t\wedge \tau_{\ell}))}-h(\check{x})
\big)^{2}\bigg)
\end{eqnarray}
We know that almost surely, $\tau_{\ell}\neq \widehat{T}$. If $\tau_{\ell}< \widehat{T}$, then for all
$i\in {1,\dots,\hat{n}}$, $\hat{z}_{i}(\hat{r}_{i}(\tau_{\ell}))=0$
and $\Psi_{\ell}(\tau_{\ell})=1$. Then we can extend $\Psi_{\ell}(t)$ for $t\geq \tau_{\ell}$, in particular for $t\geq \widehat{T}$, to be equal to $1$.
If $\tau_{\ell} > \widehat{T}$ and none of the $\hat{z}_{i}(\hat{r}_{i}(\widehat{T})$
hits $\widecheck{A}$, then $\Psi_{\ell}(t)$ has a finite limit 
in $(0,1)$ at $\widehat{T}$, and we extend $\Psi_{\ell}(t)$ for $t\geq \widehat{T}$ to be equal this limit.
If $\tau_{\ell} > \widehat{T}$ and for the value $i_{0}$, such that
$\hat{z}_{i_{0}}(\hat{r}_{i_{0}}(\widehat{T})=\hat{y}_{i_{0}}$, we have
$\hat{y}_{i_{0}}\in \widecheck{A}$, then the factor
\begin{displaymath}
\exp\bigg(
-\frac{1}{2}C^{\rm eff}_{i_{0}\hat{y}_{i_{0}}}(t\wedge \tau_{\ell})
\big(\vert\tilde{\phi}_{\hat{z}_{i_{0}}(\hat{r}_{i_{0}}(t\wedge \tau_{\ell}))}\vert +
\vert h(\hat{y}_{i_{0}})\vert+
\ell-L(t\wedge\tau_{\ell})\big)^{2}
+\frac{1}{2}C^{\rm eff}_{i_{0}\hat{y}_{i_{0}}}(t\wedge \tau_{\ell})
\big(\tilde{\phi}_{\hat{z}_{i_{0}}(\hat{r}_{i_{0}}(t\wedge \tau_{\ell}))}-h(\hat{y}_{i_{0}})
\big)^{2}\bigg),
\end{displaymath}
appearing in the product \eqref{EqDefPsiell}, tends to $0$ at 
$\widehat{T}$. Indeed, $C^{\rm eff}_{i\hat{y}_{i_{0}}}(t\wedge \tau_{\ell})$ tends to infinity and
\begin{displaymath}
\big(\vert h(\hat{y}_{i_{0}})\vert+\ell-L(\widehat{T})\big)^{2}>
h(\hat{y}_{i_{0}})^{2}.
\end{displaymath}
Thus, $\Psi_{\ell}(t)$ has limit $0$ at $\widehat{T}$. In this case we will extend 
$\Psi_{\ell}(t)$ to be equal $0$ for $t\geq \widehat{T}$.
Extended this way, $(\Psi_{\ell}(t))_{t\geq 0}$ is a continuous process on
$[0,+\infty)$, stopped at time $\tau_{\ell}\wedge\widehat{T}$. The initial value of the process is
\begin{displaymath}
\Psi_{\ell}(0)=
\prod_{\substack{ 1\leq i\leq\hat{n}\\\check{x}\in\widecheck{A}}}
\exp\bigg(
-\frac{1}{2}C^{\rm eff}_A(\hat{x}_{i},\check{x})
\big(\vert h(\hat{x}_{i})\vert +
\vert h(\check{x})\vert+\ell\big)^{2}
+\frac{1}{2}C^{\rm eff}_A (\hat{x}_{i},\check{x})
\big(h(\hat{x}_{i})-h(\check{x})\big)^{2}\bigg).
\end{displaymath}

\begin{lemma}
\label{LemMartingale}
Let $(\widehat{\mathcal{F}}_{t})_{t\geq 0}$ be the natural filtration of
$(\tilde{\phi}_{\hat{z}_{i}(\hat{r}_{i}(t))},\hat{r}_{i}(t))_{1\leq i\leq \hat{n}, t\geq 0}$. Let
$\ell > 0$. Assume that for all
$i,j\in \lbrace 1,\dots,\hat{n}\rbrace$, $h(\hat{x}_{i})h(\hat{x}_{j})\geq 0$,
and that for all $\check{x},\check{x}'\in \widecheck{A}$,
$h(\check{x})h(\check{x}')\geq 0$. Then the process $(\Psi_{\ell}(t))_{t\geq 0}$
defined previously (see \eqref{EqDefPsiell}) is a martingale with respect the filtration 
$(\widehat{\mathcal{F}}_{t})_{t\geq 0}$. In particular
\begin{eqnarray*}
 \Psi_{\ell}(0)&=&
\prod_{\substack{ 1\leq i\leq\hat{n}\\\check{x}\in\widecheck{A}}}
\exp\bigg(
-\frac{1}{2}C^{\rm eff}_A (\hat{x}_{i},\check{x})
\big(\vert h(\hat{x}_{i})\vert +
\vert h(\check{x})\vert+\ell\big)^{2}
+\frac{1}{2}C^{\rm eff}_A(\hat{x}_{i},\check{x})
\big(h(\hat{x}_{i})-h(\check{x})\big)^{2}\bigg)
\\
&=&\mathbb{P}(\tau_{\ell}< \widehat{T})+
\mathbb{E}[1_{\tau_{\ell}> \widehat{T}}\Psi_{\ell}(\widehat{T})].
\end{eqnarray*}
\end{lemma}

\begin{proof}
Since the process $(\Psi_{\ell}(t))_{t\geq 0}$ is bounded, with values in
$[0,1]$, to show that it is a martingale, we need only to show that it is a local martingale.
Note that the process is stopped at $\tau_{\ell}\wedge\widehat{T}$, which is a stopping time for the filtration $(\widehat{\mathcal{F}}_{t})_{t\geq 0}$. To show the local martingale property, we apply Itô's formula and check that
${d\Psi_{\ell}(t)}$ has no bounded variation term and only a local martingale term. The computations are straightforward but quite long; we reproduce them in the appendix.
\end{proof}

We are now ready to prove Proposition \ref {thmDist2Sets}:

\begin{proof}
We will prove the proposition by induction on the number of edges $\vert E\vert$ in the graph.
As we have already pointed out, the result in the case where $|E|=1$ follows from 
Lemma \ref {LemLevyBridge}.

Consider now that $\vert E\vert\geq 2$. As before, we can assume that there is no edge $e$ with
both endpoints in $\widehat{A}$ (as we could have removed it). 
We will reuse the notation of the beginning of Section \ref{SecExplicit} and of Lemma
\ref{LemMartingale}. We will use an exploration $(\hat{r}_{i}(t))_{1\leq i\leq \hat{n},t\geq 0}$
of $\tilde{\phi}$ such that
\begin{displaymath}
L_{1}(\hat{r}_{1}(t))=\dots =L_{\hat{n}}(\hat{r}_{\hat{n}}(t)):=L(t).
\end{displaymath}
$\widehat{T}$ is the first time one of the $\hat{r}_{i}(t)$ hits $R(\hat{e}_{i})$.
Denote $\widehat{\mathcal{F}}_{\widehat{T}}$ the sigma-algebra of
\begin{displaymath}
(\tilde{\phi}_{\hat{z}_{i}(\hat{r}_{i}(t))},\hat{r}_{i}(t))
_{1\leq i\leq\hat{n}, 0\leq t\leq \widehat{T}}.
\end{displaymath}
Let $\ell>0$, and $\tau_{\ell}$ the first time $L(t)$ hits $\ell$.
We will consider the martingale $(\Psi_{\ell}(t))_{t\geq 0}$ that appears in Lemma
\ref{LemMartingale}, according to which
\begin{equation}
\label{EqPsiST}
\Psi_{\ell}(0)=\mathbb{P}(\tau_{\ell}< \widehat{T})+
\mathbb{E}[1_{\tau_{\ell}> \widehat{T}}\Psi_{\ell}(\widehat{T})].
\end{equation}

Let $i_{0}$ be the a.s. unique value of $i$ such that
$\hat{r}_{i}(\widehat{T})=R(\hat{e}_{i})$.
Let $\underline{\mathcal{G}}=(\underline{V},\underline{E})$ be the following network:
\begin{itemize}
\item The set of vertices is $\underline{V}= (V\setminus\widehat{A})
\cup\lbrace \hat{z}_{i}(\hat{r}_{i}(\widehat{T}))\vert 1\leq i\leq \hat{n}, i\neq i_{0}\rbrace$.
\item The set of edges $\underline{E}$ is obtained by removing  the edges $(\hat{e}_{i})_{1\leq i\leq\hat{n}}$ from $E$ and then adding the
edges $(\lbrace\hat{z}_{i}(\hat{r}_{i}(\widehat{T})),\hat{y}_{i}\rbrace)_{i\neq i_{0}}$ joining
$\hat{z}_{i}(\hat{r}_{i}(\widehat{T}))$ to $\hat{y}_{i}$.
\item The conductances of edges in $\underline{E}$ that were already present in $E$ are unchanged. The conductance of an edge $\lbrace\hat{z}_{i}(\hat{r}_{i}(\widehat{T})),\hat{y}_{i}\rbrace$ is
$(R(\hat{e}_{i})-\hat{r}_{i}(\widehat{T}))^{-1}$.
\end{itemize}
The network $\underline{\mathcal{G}}$ has one less edge than $\mathcal{G}$. Denote
$\underline{\widehat{A}}$ the set 
$\lbrace\hat{z}_{i}(\hat{r}_{i}(\widehat{T}))\vert 1\leq i\leq \hat{n}\rbrace$.
$\underline{\tilde{\phi}}$ will denote the process defined on the metric graph
$\underline{\widetilde{\mathcal{G}}}$ associated to $\underline{\mathcal{G}}$ and distributed
like a metric graph GFF with boundary condition $h$ on $\widecheck{A}$ and
condition $\tilde{\phi}_{\hat{z}_{i}(\hat{r}_{i}(\widehat{T}))}$ on
$\hat{z}_{i}(\hat{r}_{i}(\widehat{T}))$, $1\leq i\leq\hat{n}$. This boundary condition has constant sign on
$\underline{\widehat{A}}$.

The event $\tau_{\ell}< \widehat{T}$ is contained inside the event 
$\delta_{\widehat{A},\widecheck{A}}(\tilde{\phi})\geq \ell$. On the event
$\tau_{\ell}>\widehat{T}, y_{i_{0}}\in\widecheck{A}$, $\Psi_{\ell}(\widehat{T})$ equals $0$ and
$\delta_{\widehat{A},\widecheck{A}}(\tilde{\phi})<\ell$.
On the event $\tau_{\ell}>\widehat{T}, y_{i_{0}}\notin\widecheck{A}$, the value of 
$\Psi_{\ell}(\widehat{T})$ is the expression \eqref{EqExplDist} applied to the network
$\underline{\mathcal{G}}$ and the distance $\ell-L(\widehat{T})$. Applying the induction
hypothesis, we get that
\begin{displaymath}
\Psi_{\ell}(\widehat{T})
=\mathbb{P}\Big(\delta_{\underline{\widehat{A}},\widecheck{A}}(\underline{\tilde{\phi}})
\geq \ell-L(\widehat{T})\Big\vert\widehat{\mathcal{F}}_{\widehat{T}}\Big).
\end{displaymath}
Since by construction 
$\delta_{\widehat{A},\widecheck{A}}(\tilde{\phi})=
L(\widehat{T})+
\delta_{\underline{\widehat{A}},\widecheck{A}}(\underline{\tilde{\phi}})$, 
applying \eqref{EqPsiST} we get that
\begin{displaymath}
\Psi_{\ell}(0)=\mathbb{P}(\delta_{\widehat{A},\widecheck{A}}(\tilde{\phi})\geq \ell),
\end{displaymath}
which concludes the proof of the proposition.
\end{proof}

\subsection {Distributions related to pseudo-metric balls and variants}

Next we consider the first {passage} set 
$\widetilde{\Lambda}_{a}$ which is the set of points $x$ in the metric graph, for which there exists a continuous path joining $x$ to $\widehat{A}$ such that
$\tilde{\phi}(z)\geq a$ along the path. In the case where $\widehat A$ is just one point, one can view this set as a pseudo-metric ball. 

In the same framework as before, let us first assume that the boundary set $A\subseteq V$ contains at least two vertices and we consider a partition
of $A$ in two non-empty subsets $\widehat{A},\widecheck{A}$.

Let be $a\in\mathbb{R}$, $a<\min_{\widehat{A}} h$. Let $C^{\rm eff}(\Lambda_{a}, \widecheck{A})$ be the effective conductance between $\widetilde{\Lambda}_{a}$ and $\widecheck{A}$ 
(finite if and only if $\widetilde{\Lambda}_{a}\cap \widecheck{A}=\emptyset$)
and let $R^{\rm eff}(\widetilde{\Lambda}_{a}, \widecheck{A})$ be its inverse, the effective resistance.

We first assume that $h$ is constant on $\widecheck{A}$. Note that this is equivalent to the fact that $\widecheck A$ is a singleton (when one identifies all points of $\widecheck A$).  Let
\begin{displaymath}
m:=C^{\rm eff}(\widehat{A},\widecheck{A})^{-1}
\sum_{\substack{\hat{x}\in\widehat{A}\\\check{x}\in\widecheck{A}}}C^{\rm eff}_A(\hat{x},\check{x})
h(\hat{x}).
\end{displaymath}

\begin{proposition}
\label{ThmReffLocSet}
The Laplace transform of the random variable $C^{\rm eff} (\widetilde \Lambda_a, \widecheck A)$ in $[0,\infty]$ is given by  (for $u >0$),
$$\mathbb{E}\left[e^{-u C^{\rm eff}(\widetilde{\Lambda}_{a}, \widecheck{A})}\right]=
\exp\left(-\frac{1}{2}C^{\rm eff}(\widehat{A},\widecheck{A})
\left[ \left(m-a+\sqrt{(h ( \widecheck A)-a)^{2}+2u}\right)^{2}
 - (m-h ( \widecheck A))^{2} \right]\right).
$$
\end{proposition}

\begin{proof}
We consider the pseudo-metric $\delta_{x,y}(\tilde{\phi}-a)$associated to the GFF
$\tilde{\phi}-a$. Then, by construction of 
$\widetilde{\Lambda}_{a}$,
\begin{displaymath}
\delta_{\widehat{A},\widecheck{A}}(\tilde{\phi}-a)=
\delta_{\widetilde{\Lambda}_{a},\widecheck{A}}(\tilde{\phi}-a).
\end{displaymath}
Using the fact that $\widetilde{\Lambda}_{a}$ is a local set for $\tilde{\phi}$ and applying Proposition \ref{thmDist2Sets}, we get that for all positive $\ell$, 
$$
\mathbb{P}\left(\delta_{\widetilde{\Lambda}_{a},\widecheck{A}}(\tilde{\phi}-a)\geq \ell
\Big\vert \widetilde{\Lambda}_{a}, \tilde{\phi}_{\vert \widetilde{\Lambda}_{a}}\right)=
\exp\left(-\frac{1}{2}C^{\rm eff}(\widetilde{\Lambda}_{a}, \widecheck{A})
 \left[ (h ( \widecheck A)-a+\ell)^{2}
- (h ( \widecheck A)-a)^{2} \right] \right).
$$
In above identity we have to distinguish two cases. 
Either $\widetilde{\Lambda}_{a}\cap \widecheck{A}=\emptyset$ and then $\tilde{\phi}$ equals $a$ on
$\partial\widetilde{\Lambda}_{a}\setminus\widehat{A}$. Or 
$\widetilde{\Lambda}_{a}\cap \widecheck{A}\neq\emptyset$ and then
$C^{\rm eff}(\widetilde{\Lambda}_{a}, \widecheck{A})=+\infty$ and both sides of the equality equal $0$.

Moreover,
\begin{eqnarray*}
\lefteqn { \mathbb{E}\left[\mathbb{P}\left(\delta_{\widetilde{\Lambda}_{a},\widecheck{A}}(\tilde{\phi}-a)\geq \ell
\Big\vert \widetilde{\Lambda}_{a}, \tilde{\phi}_{\vert \widetilde{\Lambda}_{a}}\right)\right]} \\
&=&
\mathbb{P}\left(\delta_{\widehat{A},\widecheck{A}}(\tilde{\phi}-a)\geq \ell\right)
\\&=&
\prod_{\substack{\hat{x}\in\widehat{A}\\\check{x}\in\widecheck{A}}}
\exp\left(-\frac{1}{2}C^{\rm eff}_A(\hat{x},\check{x})
(h(\hat{x})-a+\vert h ( \widecheck A)-a\vert+\ell)^{2}
+\frac{1}{2}C^{\rm eff}_A(\hat{x},\check{x})(h(\hat{x})-h ( \widecheck A))^{2}\right)
\\&=&
\exp\left(-\frac{1}{2}C^{\rm eff}(\widehat{A},\widecheck{A})
(m-a+\vert h ( \widecheck A)-a\vert+\ell)^{2}
+\frac{1}{2}C^{\rm eff}(\widehat{A},\widecheck{A})(m-h ( \widecheck A))^{2}\right)
.
\end{eqnarray*}
By taking
\begin{displaymath}
u=\frac{1}{2}(\vert h ( \widecheck A)-a\vert+\ell)^{2}-
\frac{1}{2}(h ( \widecheck A)-a)^{2},
\end{displaymath}
we get the proposition.
\end {proof}

Note that one can rephrase the proposition by saying that $R^{\rm eff}(\widetilde{\Lambda}_{a}, \widecheck{A})$ is distributed like the {\em last} visit time of level  $a$ 
by a Brownian bridge of length
$R^{\rm eff}(\widehat{A}, \widecheck{A})$ from
 $h ( \widecheck A)$ to $m$ (with the convention that this last time is $0$ when the bridge does not visit $a$, which corresponds to the event that $\widetilde \Lambda_a$ intersects $\widecheck A$).
Indeed, the expression of the Laplace transform  shows that the distribution of $C^{\rm eff} (\widetilde \Lambda_a, \widecheck A)$ (and therefore also the distribution of its inverse) depends only on the four parameters
$m$, $h ( \widecheck A)$, $C^{\rm eff}(\widehat{A}, \widecheck{A})$ and $a$. The distribution of $R^{\rm eff} ( \widetilde \Lambda_a, \widecheck A)$  is therefore the same as for the GFF on a graph consisting of just
one edge, i.e. for a Brownian bridge of length $R^{\rm eff}(\widehat{A}, \widecheck{A})$ from 
$m$ to $h ( \widecheck A)$. 

This can be used to obtain estimates for the discrete GFF on $\mathcal{G}$. One can for instance define
\begin{displaymath}
\Lambda_{a}:=\left\lbrace x\in V\Big\vert
\exists \gamma~\text{discrete nearest neighbour path from}~x~
\text{to}~\widehat{A},~\text{s.t.}~\inf_{\gamma\setminus\lbrace x\rbrace}\phi\geq a\right\rbrace.
\end{displaymath}
$\Lambda_{a}$ is an optional set for the discrete GFF $\phi$. $\phi$ is greater on equal to
$a$ on $\Lambda_{a}$, except on neighbours of $V\setminus\Lambda_{a}$, that is to say on the boundary
of $\Lambda_{a}$. Indeed, one has to discover a vertex $x$ with
$\phi_{x}<a$ to know where to stop. By construction,
$
\widetilde{\Lambda}_{a}\subseteq\Lambda_{a}$
and 
$
R^{\rm eff}(\Lambda_{a},\widecheck{A})\leq R^{\rm eff}(\widetilde{\Lambda}_{a},\widecheck{A})$.
Hence, if $a\leq\min_{\widehat{A}} h$ and $h$ is constant on $\widecheck{A}$, then the distribution of
$R^{\rm eff}(\Lambda_{a},\widecheck{A})$ is stochastically dominated by the law given by Proposition \ref{ThmReffLocSet}.

\medbreak

Let us now consider the following consequence of Proposition \ref {ThmReffLocSet}: This time, we consider $\widecheck A$ to be empty, so that 
$(\widetilde{\Lambda}_{a})$ starts from {\em the whole boundary $A$}. We however assume that $A$ and $V \setminus A$ are not empty. 

\begin{corollary}
\label{CorDistrib1ptandHitting}
(i) 
Let $x_{0}\in V\setminus A$ and $m=\mathbb{E}[\tilde{\phi}_{x_{0}}]$.
Let $(W_{t})_{t\geq 0}$ be a standard Brownian motion starting from $m$.
Define
\begin{displaymath}
I_{t}:=\inf_{[0,t]}W,
\qquad
T_{a}:=\inf\lbrace t\geq 0\vert W_{t}=a\rbrace .
\end{displaymath}
The joint distribution of
\begin{displaymath}
\bigg(\tilde{\phi}_{x_{0}},
(R^{\rm eff}(x_{0},A)-R^{\rm eff}(x_{0},\widetilde{\Lambda}_{a}))_{a\leq \min_{A}h},
\sup_{\substack{\gamma~\text{continuous}\\ \text{path connecting}\\x_{0}~\text{to}~A}} 
\min_{\gamma}\tilde{\phi}\wedge \min_{A}h\bigg)
\end{displaymath}
is the same as the joint distribution of
\begin{displaymath}
\left(W_{R^{\rm eff}(x_{0},A)},
(T_{a}\wedge R^{\rm eff}(x_{0},A))_{a\leq \min_{A}h},
I_{R^{\rm eff}(x_{0},A)}\wedge\min_{A}h\right).
\end{displaymath}

(ii) Consider now that the boundary condition $h$ is non-negative and define the sets
\begin{displaymath}
B(A,\ell):=\left\lbrace z\in\widetilde{\mathcal{G}}\Big\vert
\delta_{z,A}\leq \ell
\right\rbrace,~\ell\geq 0.
\end{displaymath}
The joint distribution of
\begin{displaymath}
\bigg(\vert\tilde{\phi}_{x_{0}}\vert - \delta_{x_{0},A},
(R^{\rm eff}(x_{0},A)-R^{\rm eff}(x_{0},B(A,\ell)))_{\ell\geq 0},
- \delta_{x_{0},A}\bigg)
\end{displaymath}
is the same as the joint distribution of
\begin{displaymath}
\left(W_{R^{\rm eff}(x_{0},A)},
(T_{-\ell}\wedge R^{\rm eff}(x_{0},A))_{\ell\geq 0},
I_{R^{\rm eff}(x_{0},A)}\wedge 0\right).
\end{displaymath}
\end{corollary}

\begin{proof}
(i) First of all $\tilde{\phi}_{x_{0}}$ is distributed like $W_{R^{\rm eff}(x_{0},A)}$.
One needs to check that for any family $a_{k}<\dots<a_{2}<a_{1}\leq h$, the distribution of
\begin{displaymath}
(R^{\rm eff}(x_{0},A)-R^{\rm eff}(x_{0},\widetilde{\Lambda}_{a_{i}}))_{1\leq i\leq k}
\end{displaymath}
conditionally on $\tilde{\phi}_{x_{0}}$ is the same as the distribution
of $(T_{a_{i}}\wedge R^{\rm eff}(x_{0},A))_{1\leq i\leq k}$ conditionally on
$W_{R^{\rm eff}(x_{0},A)}$. The conditional distribution of 
$R^{\rm eff}(x_{0},A)-R^{\rm eff}(x_{0},\widetilde{\Lambda}_{a_{1}})$ follows from Proposition \ref{ThmReffLocSet}.
Then, for the conditional distribution of 
$R^{\rm eff}(x_{0},\widetilde{\Lambda}_{a_{1}})-R^{\rm eff}(x_{0},\widetilde{\Lambda}_{a_{2}})$
given $\tilde{\phi}_{x_{0}}$ and $R^{\rm eff}(x_{0},A)-R^{\rm eff}(x_{0},\widetilde{\Lambda}_{a_{1}})$, one can
iterate, using the fact that $\widetilde{\Lambda}_{a_{2}}\setminus\widetilde{\Lambda}_{a_{1}}$ is the first {passage} set of level
$a_{2}$ on $\widetilde{\mathcal{G}}\setminus\widetilde{\Lambda}_{a_{1}}$, and so on.

(ii) Combine (i) with Proposition \ref{thmGeneralisedLevy}.
\end{proof}

\section{Predictions for the two-dimensional continuum GFF}
\label{SecConj2DGFF}
\label {S4}
\subsection {CLE(4)-distance, and conjectures}

Let us briefly survey some relevant background on the two-dimensional GFF and the conformal loop ensembles CLE$_4$: 

When $D$ is an open domain in the plane with non-polar boundary, one can define the continuum GFF $\phi_{D}$ in $D$ with zero boundary conditions, associated to the Dirichlet from $ \int_D | \nabla f|^2$. Mind that, as opposed to the GFF defined on metric graphs, this GFF is not a random function anymore (it is just a random generalised function -- see eg. \cite {SheffieldGFF,Wln} for background). 

When $D$ is simply connected, one can define (see \cite{SheffieldWerner2012CLE}) a special random countable collection of Jordan loops in $D$ which do not touch each other and do not surround each other. This collection of loops (whose law is invariant under conformal transformations of the domain) is  called CLE$_4$ and can be coupled with the continuous GFF in two natural yet different ways: 
\begin {itemize}
\item 
One can couple the Gaussian free field $\phi_{D}$ and the $\hbox{CLE}_{4}$ loop ensemble as follows (this coupling was pointed out by Miller and Sheffield \cite{MillerSheffieldCLE4GFF}).
Given the {CLE}$_{4}$, for each loop $\Gamma$ sample in the domain $\hbox {Int} ( \Gamma)$ surrounded by $\Gamma$ an independent Gaussian free field $\phi_{\Gamma}$ with zero boundary condition on $\Gamma$. 
Also sample an independent uniformly distributed sign $\sigma_{\Gamma}\in\lbrace -1,1\rbrace$. Then, for a well-chosen constant $\lambda >0$ (equal to $\sqrt{\pi / 8}$ in our normalisation of the GFF), the field
\begin{displaymath}
\sum_{\Gamma\in\hbox{CLE}_{4}} 1_{\hbox{Int}(\Gamma)}(\phi_{\Gamma}+\sigma_{\Gamma} 2 \lambda)
\end{displaymath}
is a Gaussian free field $\phi_{D}$ in $D$ with zero boundary conditions 
(see also \cite{WangWu2014GFF1,MillerWatsonWilson2015NestField}). In this coupling the 
$\hbox{CLE}_{4}$ loops can be interpreted as level lines of the 
continuum Gaussian free field $\phi_{D}$ on $D$ and are in fact deterministic functions of the GFF (\cite{MillerSheffieldCLE4GFF}, see also \cite {AruSepulvedaWerner2016BoundedLocSets}). 
The discontinuity $2\lambda$ is often referred to as  the
\textit{height gap} (this was introduced by Schramm-Sheffield, 
see \cite{SchrammSheffield2013ContourContGFF}, Section 4.4 and
\cite{SchrammSheffield2009ContourDiscrGFF}, Section 1.2).
\item
There is another coupling between {CLE}$_{4}$ and the GFF $\phi_{D}$ pointed out by Sheffield, Watson and Wu (see \cite{WuThesis,SheffieldWatsonWuMetricCLE4}), which is based on the 
conformal invariant growing mechanism in {CLE}$_{4}$ constructed in 
\cite{WernerWu2013ExplorCLE}, Section 4. This conformal invariant growing mechanism associates to each loop 
$\Gamma$ in $\hbox{CLE}_{4}$ a time parameter $t(\Gamma)>0$, where the loops closer to the boundary 
$\partial D$ of the domain tend to have smaller time parameters, and the loops farther away tend to have larger time parameters. Then one samples for each loop $\Gamma$ in the interior surrounded by $\Gamma$ an independent GFF $\phi_{\Gamma}$ with zero boundary condition at $\Gamma$, which is also independent from the time parameters. The field
\begin{displaymath}
\sum_{\Gamma\in\hbox{CLE}_{4}} 1_{\hbox{Int}(\Gamma)}(\phi_{\Gamma}+2 \lambda-t(\Gamma))
\end{displaymath}
is distributed as the Gaussian free field in $D$ with $0$ boundary condition (see Theorem 1.2.2 in \cite{WangWu2014GFF1} and the references therein).
\end {itemize}

\begin{figure}[ht!]
  \centering 
  \includegraphics[width=.9\textwidth]{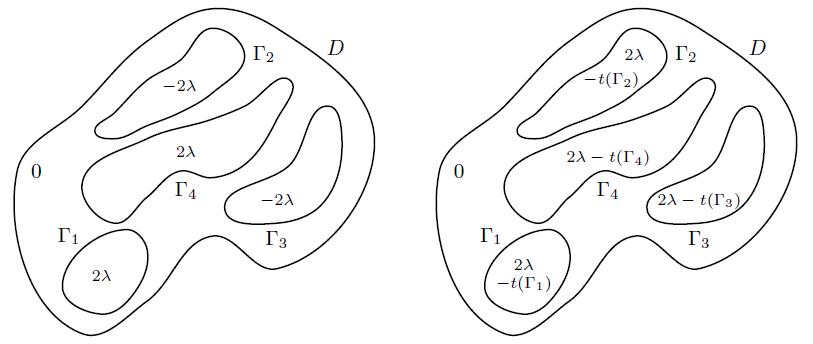}
  \caption{
Two couplings between {CLE}$_{4}$ and the GFF. Only four {CLE}$_{4}$ loops are represented. 
}
  \label{fig2Coupl}
\end{figure}

These couplings are reminiscent of the two ways to construct a one-dimensional Brownian motion (or rather a one-dimensional Brownian bridge) from a reflected Brownian motion (or a reflected Brownian bridge). The first one corresponds to tossing an independent coin to decide the sign of each excursion away from the origin by a reflected Brownian motion, while the second 
is reminiscent of L\'evy's theorem. In both cases, the set of points that are not surrounded by any $\hbox{CLE}_{4}$ loop plays the role of the zero-set of the Brownian motion (or bridge). 

While the notion of zero-set of the continuous GFF is tricky (see the previous notion of height-gap), the zero-set of the GFF on a metric graph is easy to define. As explained in 
\cite{Lupu2014LoopsGFF,Lupu2015ConvCLE}, when one approximates the continuous domain $D$ and the GFF on it, by a discrete fine-mesh lattice and the GFF on its metric graph, then the outermost boundaries of the excursion-domains away from $0$ in the metric graph do indeed converge to the CLE$_4$ loops associated to the GFF, in the fine-mesh limit. This means that in some sense, the first coupling of CLE$_4$ with the GFF can be viewed as a limit of a similar coupling on the metric graph (see \cite {Lupu2015ConvCLE} for details). 

It is therefore very natural to conjecture that the second coupling can be obtained in the same way as the scaling limit of our pseudo-metric: 
\begin{conjecture}
Let $D$ be an open bounded simply connected domain of $\mathbb{C}$.
Let $D^\varepsilon$ be a fine-mesh approximation of the domain $D$, $A^{\varepsilon}$ the approximation of the boundary $\partial D$ and $\tilde{\phi}^{\varepsilon}$ the metric graph GFF on $D^\varepsilon$ with zero boundary condition on $A^{\varepsilon}$. When $\varepsilon\to 0$, the local time distance 
$\delta_{x,A}(\tilde{\phi}^{\varepsilon})$, jointly with the outermost sign clusters of 
$\tilde{\phi}^{\varepsilon}$, converge to the time parameters $t(\Gamma)$ on CLE$_{4}$.
\end{conjecture}

Actually, in the paper in preparation \cite {SheffieldWatsonWuMetricCLE4}, Sheffield, Watson and Wu show  that the collection $t(\Gamma)$ can be extended to a conformally invariant metric on the collections of 
$\hbox{CLE}_{4}$ loops in $D$ together with $\partial D$, i.e. that one can also define in a similar way the distance between two CLE$_4$ loops ($t(\Gamma)$ is the distance of the loop $\Gamma$ to $\partial D$). This is 
related to the inversion property of CLE$_4$ (see e.g. \cite{KempainenWerner2015NestedCLE}), which makes it possible to define for each loop $\Gamma_0$ the time $t_{\Gamma_0} (\Gamma)$. The main point
in \cite {SheffieldWatsonWuMetricCLE4} is then to check that $t_\Gamma (\Gamma_0) = t_{\Gamma_0} ( \Gamma)$, which is not  obvious at all from the growing mechanism point of view. 
However, this symmetry is obvious for the metric graph analogue (the distance between two excursion of the GFF away from the origin to be the minimal local time at $0$ collected by a continuous path in the metric graph that joins them, which 
is clearly symmetric). Hence, establishing the previous conjecture would provide an alternative proof of the existence of this random metric on CLE$_4$ loops.
 
To further substantiate this conjecture, one can note that in the scaling limit, the behaviour of the first {passage} sets described in Corollary \ref{CorDistrib1ptandHitting} 
is known to correspond to the corresponding quantity for the two-dimensional continuum GFF
in a simply connected domain (see for instance \cite{AruSepulvedaWerner2016BoundedLocSets}). Let $D\neq \mathbb{C}$ be such an open simply connected domain, let $\phi_{D}$ 
be the continuum GFF inside $D$ with zero boundary condition and let $z_{0}\in D$. Denote $2\lambda$ the
height gap. Given $a\in\mathbb{R}$,
one can define a random subset $\Lambda_{a}^{D}(z_{0})$ of $D$, such that
\begin{itemize}
\item $D\setminus \Lambda_{a}^{D}(z_{0})$ is an open simply connected subset of $D$ containing $z_{0}$,
\item $\Lambda_{a}^{D}(z_{0})$ is measurable with respect the GFF $\phi_{D}$,
\item conditionally on $\Lambda_{a}^{D}(z_{0})$ and $\phi_{D}$ on $\Lambda_{a}^{D}(z_{0})$, 
the restriction of $\phi_{D}$ to $D\setminus \Lambda_{a}^{D}(z_{0})$ is distributed like a GFF with boundary condition $a$,
\item $\Lambda_{a}^{D}(z_{0})$ is minimal for the above properties, that is to say any other random set satisfying the above three properties a.s. contains $\Lambda_{a}^{D}(z_{0})$.
\end{itemize}
The distribution of $\operatorname{crad}(z_{0},D\setminus \Lambda_{a}^{D}(z_{0}))$, 
the conformal radius of
$D\setminus \Lambda_{a}^{D}(z_{0})$ seen from $z_{0}$ is explicitly known. And it is known that the random variable 
\begin{equation}
\label{EqDeltaLogCrad}
\log\left(\dfrac{\operatorname{crad}(z_{0},D)}
{\operatorname{crad}(z_{0},D\setminus \Lambda_{a}^{D}(z_{0}))}\right)
\end{equation}
is distributed like the first hitting time of level ${a \pi}/ ({2\lambda})$ by a standard Brownian motion starting from $0$ (see eg. Proposition 15 in \cite{AruSepulvedaWerner2016BoundedLocSets}).
Moreover, \eqref{EqDeltaLogCrad} is the limit of
\begin{displaymath}
\dfrac{\pi^{2}}{4\lambda^{2}}(R^{\rm eff}(z_{0},\partial D^{\varepsilon})-
R^{\rm eff}(z_{0},\widetilde{\Lambda}_{a}^{\varepsilon}))
\end{displaymath}
in an approximation of $D$ by a metric graph.

\subsection {Background on Riemann surfaces and extremal distance}

We will explain why our results should also provide (in the scaling limit) some new formulas related to the GFF, loop-soups and CLE$_4$-type models on Riemann surfaces. Recall that 
on such surfaces, the SLE approach is much less obvious (when one cuts a Riemann surface open via a slit, one changes its modulus) and also that the notion of ``outside'' and ``inside'' of a cluster is also problematic. However, the scaling limit of quantities like the effective resistance still make sense, as we now explain. 

A first remark is that in two dimensions and in the scaling limit, the invariance under electric rewiring of the boundary excursion kernel is closely related to conformal invariance. 
Indeed, if $D$ denotes a (non-necessarily simply connected) open  domain in the complex plane (this definition is also immediate in Riemann surfaces), then one can define the excursion measure away from its boundary, and see 
that this excursion measure is conformally invariant (see \cite {LawlerWerner2000}, where this fact was probably first used in the context of conformally invariant structures such as percolation). This indicates that the quantities involving the effective conductance matrix (i.e. boundary excursion kernel) have nice natural scaling limits when one lets the mesh of the lattice go to $0$.

Let us first in the present subsection quickly browse through background on Riemann surfaces (we refer to \cite{FarkasKra1992RiemSurf},
\cite{Ahlfors2010ConfInv}, Sections 9 and 10,
\cite{Jurchescu1961BordRiemSurf} and \cite{Dubedat2015Virasoro1}, Sections 2 and 3 for details) and on extremal distance.
Let $\Sigma$ be a compact \textit{bordered Riemann surface} (\cite{Jurchescu1961BordRiemSurf}). 
This is a smooth compact connected surface equipped with a complex structure given by a complex analytic atlas. The local charts model $\Sigma$  on the unit disk (for interior points)
or on the semi-disk $\lbrace z\in\mathbb{C}\vert \Im(z)\geq 0,\vert z\vert <1\rbrace$ (for boundary points, and this does in fact characterise the boundary points, the border $B(\Sigma)$ will be the set of boundary points). We assume that $B(\Sigma)$ is non-empty, and $B(\Sigma)$ has then finitely many connected components, each of which is homeomorphic to a circle.
The complex structure of $\Sigma$ induces a section $J$ of the endomorphisms of the tangent bundle of 
$\Sigma\setminus B(\Sigma)$ such that $J^{2}=- Id$.

Let $g$ be a smooth metric on $\Sigma$, compatible with its complex structure, that is to say for any tangent
vector $v$, $v$ and $Jv$ are $g$-orthogonal.
$\nabla^{g}$ will denote the gradient associated to $g$,
$\Delta_{g}$ the \textit{Laplace-Beltrami operator}, which in local coordinates is expressed as
\begin{displaymath}
\Delta_{g}=\dfrac{1}{\sqrt{\det g}}\sum_{i,j\in\lbrace 1,2\rbrace}
\dfrac{\partial}{\partial x_{i}} g^{ij} \sqrt{\det g} \dfrac{\partial}{\partial x_{j}}.
\end{displaymath}
$d\nu_{g}$ will denote the volume form associated with $g$, which in local coordinates is
\begin{displaymath}
d\nu_{g}=\sqrt{\det g}   \ ( dx_{1}\wedge dx_{2}).
\end{displaymath}
Given any other smooth metric $g'$ compatible with $\Sigma$, $g'$ and $g$ are conformally equivalent, that is to say there is a smooth function $\rho$ such that $g'=e^{\rho}g$. Moreover
\begin{displaymath}
\nabla^{g'}=e^{-\rho}\nabla^{g}, \qquad
\Delta_{g'}=e^{-\rho}\Delta_{g}, \qquad
d\nu_{g'}=e^{\rho} d\nu_{g}.
\end{displaymath}
In particular, the set of $\Delta_{g}$-harmonic functions and the Dirichlet form
\begin{equation}
\label{EqDirichletSigma}
\mathcal{E}_{\Sigma}(f,f)=\int_{\Sigma}g(\nabla^{g}f,\nabla^{g}f) d\nu_{g}
\end{equation}
do not depend on the particular choice of a metric compatible with the complex structure of $\Sigma$.

Given $K_{1}$ and $K_{2}$ two compact non-polar subsets of $\Sigma$, the {\em extremal distance}
$\operatorname{ED}(K_{1},K_{2})$ between $K_{1}$ and $K_{2}$ in $\Sigma$ is defined as follows. First, for a given metric $g$ that is 
compatible with the complex structure of $\Sigma$, one can define the distance $d_g (K_1, K_2)$ as the infimum over all 
continuous smooth paths joining $K_1$ to $K_2$ of the $\rho$-length of this path. 
Then, $\operatorname{ED}(K_{1},K_{2})$ is the supremum over all $g$'s that are compatible with the complex structure of $\Sigma$ of 
$$ 
\dfrac{d_g (K_1, K_2))^{2}}{\nu_{g}(\Sigma\setminus (K_{1}\cup K_{2}))}
$$
(this quantity is also known as the extremal length of the set of all paths connecting $K_1$ to $K_2$ in $\Sigma$).

Let us quickly review some of the properties of extremal distance: 
\begin {itemize}
 \item 
By definition, it is invariant under conformal transformations.
\item
It is also the inverse of the Dirichlet energy of the  the harmonic function $u$ on
$\Sigma\setminus (K_{1}\cup K_{2})$, with Dirichlet boundary condition $0$ on $K_{1}$, $1$ on
$K_{2}$ and zero Neumann boundary condition on
$B(\Sigma)\setminus (K_{1}\cup K_{2})$ (the normal derivative vanishes),
\begin{equation}
\label{EqELEnergy}
\operatorname{ED}(K_{1},K_{2})=\mathcal{E}_{\Sigma}(u,u)^{-1}
\end{equation}
(see \cite{Ahlfors2010ConfInv}, Section 4.9, for a proof in case $\Sigma$ is a domain of 
$\mathbb{C}$). 
\item When one  approximates $\Sigma$ by electrical networks $\Sigma^{\varepsilon}$ and
$K_{1}$ and $K_{2}$ by subsets of vertices $K_{1}^{\varepsilon}$ and $K_{2}^{\varepsilon}$, then
the extremal distance $\operatorname{ED}(K_{1},K_{2})$ is the limit as $\varepsilon$ tends to $0$
of the effective resistances $ R^{\rm eff}(K_{1}^{\varepsilon},K_{2}^{\varepsilon})$. 
Indeed, the effective resistances have a representation similar to \eqref{EqELEnergy}:
\begin{displaymath}
R^{\rm eff}(K_{1}^{\varepsilon},K_{2}^{\varepsilon})^{-1} 
=C^{\rm eff}(K_{1}^{\varepsilon},K_{2}^{\varepsilon})
=  \sum_{e~\text{edge}} \left( C(e)(u_{\varepsilon}(e^{+})-u_{\varepsilon}(e^{-}))^{2}\right),
\end{displaymath}
where $e^{+}$ and $e^{-}$ denote the end-vertices of an edge $e$ and $u_{\varepsilon}$ is a harmonic function on the vertices of $\Sigma^{\varepsilon}$, with boundary condition $0$ on $K_{1}^{\varepsilon}$ and $1$ on $K_{2}^{\varepsilon}$ (see also \cite{Duffin1962ELRes} for an interpretation of the effective resistance as the extremal
distance of a network with a similar representation).
\item 
Related to the previous fact and to the fact that the excursion measures of the Brownian motions on the metric graphs $\Sigma^{\varepsilon}$ converge to the excursion measure of Brownian motion in $\Sigma$ (away from their respective Dirichlet boundaries -- the Brownian motions are reflected on the Neumann boundary), we see that one can express the extremal distance in terms of the inverse of the total mass of the set of Brownian excursions that start on $\Sigma_1$ and end on $\Sigma_2$. This total mass is the scaling limit of the effective conductance while extremal distance is the scaling limit of effective resistance. 
\end {itemize}

\subsection {Conjectures for some GFF fuctionals on Riemann surfaces}

Consider now a partition of the border $B(\Sigma)$ of the Riemann surface $\Sigma$ into three parts $B_{1}(\Sigma),B_{2}(\Sigma),B_{3}(\Sigma)$. Here 
 $B_{3}(\Sigma)$ is allowed to be empty but not  $B_{1}(\Sigma)$ and $B_{2}(\Sigma)$. We suppose further that each of the
$B_{i}(\Sigma)$, $i\leq i\leq 3$, contains finitely many connected components and that
$B_{1}(\Sigma)$ and $B_{2}(\Sigma)$ are at positive distance of each other, i.e. their topological closures do not intersect.
To fix ideas, the reader can think of the following two simple examples:
\begin{itemize}
\item $\Sigma$ is a rectangle $Q_{R}=[0,1]\times[0,R]$,
$B_{1}(Q_{R})=[0,1]\times \lbrace 0\rbrace$,
$B_{2}(Q_{R})=[0,1]\times \lbrace R\rbrace$ and
$B_{3}(Q_{R})=\lbrace 0,1\rbrace\times (0,R)$.
\item $\Sigma$ is the annulus $
\mathcal{A}_{R}=\lbrace z\in\mathbb{C}\vert 1\leq \vert z\vert\leq R\rbrace$, 
$B_{1}(\mathcal{A}_{R})$ is the inner circle, $B_{2}(\mathcal{A}_{R})$ is the outer circle and
$B_{3}(\mathcal{A}_{R})$ is empty.
\end{itemize}
We now consider $\phi_{\Sigma}$ the Gaussian free field on $\Sigma$ associated to the Dirichlet form
\eqref{EqDirichletSigma}, with zero boundary condition on $B_{1}(\Sigma)\cup B_{2}(\Sigma)$ and
free boundary condition on $B_{3}(\Sigma)$ (recall that it is a random generalised function but not a random function).
The Wick square $:\phi_{\Sigma}^{2}:$ of $\phi_{\Sigma}$ is then also a random generalised function that can be  defined via the limit of regularisations
$\phi_{\Sigma,\varepsilon}$ of $\phi_{\Sigma}$ as 
\begin{displaymath}
:\phi_{\Sigma}^{2}: =
\lim_{\varepsilon \rightarrow 0}  \left [ \phi_{\Sigma,\varepsilon}^{2} - \mathbb{E}[\phi_{\Sigma,\varepsilon}^{2}] \right] 
\end{displaymath}
(recall that it is also a random generalised function).

The field $\phi_{\Sigma}$ is related to a Brownian loop-soup. Let $g$ be a smooth metric compatible with the complex
structure of $\Sigma$ and consider the Brownian motion on $\Sigma$ with generator $\Delta_{g}$, killed on
the border $B_{1}(\Sigma)\cup B_{2}(\Sigma)$ and instantaneously {normally} reflected on
$B_{3}(\Sigma)$. Let $p^{g}_{t}(z_{1},z_{2})$ be the associated transition densities with respect the measure
$d\nu_{g}(z_{2})$ and $\mathbb{P}^{g}_{t,z_{1},z_{2}}$ the bridge probabilities, with the conditioning
not to hit $B_{1}(\Sigma)\cup B_{2}(\Sigma)$. Then we can define an infinite measure on time-parametrised loops inside $\Sigma$ by
\begin{displaymath}
\mu^{\rm loop}_{g}(\cdot):=
\int_{\Sigma}\int_{0}^{+\infty}
\mathbb{P}^{g}_{t,z,z}(\cdot) p^{g}_{t}(z,z) \dfrac{dt}{t} d\nu_{g}(z).
\end{displaymath}
If $g'$ is another conformally equivalent metric, then $\mu^{\rm loop}_{g'}$ is the image of 
$\mu^{\rm loop}_{g}$ under a change of the root (starting and endpoint) and a change of parametrisation of loops, but the measure induced on the set occupied by a loop is the same 
(\cite{Dubedat2015Virasoro1}, Section 3).
See \cite{LawlerLimic, LeJan2011Loops, LeJanMarcusRosen2012Loops} for the definition of the loop measure of a Markov process in a wider framework.

Let $\mathcal{L}_{g,1/2}$ be the Poisson point process of intensity $\mu^{\rm loop}_{g } /2$.
We see it as a random infinite countable collection of Brownian loops. It is a \textit{Brownian loop-soup}
after the terminology of \cite{LawlerWerner2004ConformalLoopSoup}. The intensity parameter
${1}/{2}$ corresponds to the central charge $c=1$ (\cite{Lawler2009PartFuncSLELoopMes}).
The law of $\mathcal{L}_{g,{1}/{2}}$ is invariant, up to rerooting and reparametrisation, under conformal
transformations of $\Sigma$. The centred occupation field of $\mathcal{L}_{g, {1}/{2}}$,
$L^{\rm ctr}(\mathcal{L}_{g,{1}/{2}})$, is defined as follows:
\begin{displaymath}
L^{\rm ctr}(\mathcal{L}_{g,{1}/{2}})(f):=
\lim_{\varepsilon\rightarrow 0} \left( 
\sum_{\gamma\in\mathcal{L}_{g,{1}/{2}}, t_{\gamma}>\varepsilon}
\int_{0}^{t_{\gamma}}f(\gamma(s)) ds
-\mathbb{E}\bigg[
\sum_{\gamma\in\mathcal{L}_{g,{1}/{2}}, t_{\gamma}>\varepsilon}
\int_{0}^{t_{\gamma}}f(\gamma(s)) ds
\bigg] \right),
\end{displaymath}
where $t_{\gamma}$ is the lifetime of the loop $\gamma$.
It turns out that this centered field $L^{\rm ctr}(\mathcal{L}_{g,{1}/{2}})$ is distributed exactly as 
\begin{displaymath}
\dfrac{1}{2}:\phi_{\Sigma}^{2}: d\nu_{g}.
\end{displaymath}
For the existence of the centered occupation field and the equality in law on an open subset of $\mathbb{C}$,
see \cite{LeJan2011Loops}, chapter 10.

The following informal description could be helpful: The field $\phi_{\Sigma}$ is non-zero and has a constant sign on each loop
of $\mathcal{L}_{g,{1}/{2}}$. The field $\frac{1}{2}\phi_{\Sigma}^{2} d\nu_{g}$ is the occupation field of the loops. The zero set of $\phi_{\Sigma}$ can be viewed as
the complementary of the set of points visited by the loops of 
$\mathcal{L}_{g,{1}/{2}}$. Since $\phi_{\Sigma}$ is a generalised function, this intuition needs to be made more precise (see also \cite {QW}). However this relation between the loops and the GFF is exact on 
a metric graph (\cite{Lupu2014LoopsGFF}).

\medbreak

Let us now list three examples (more general variants are easy to obtain along the same lines) of concrete conjectures in the present setup, that should correspond to the scaling limit of the formulas derived on metric graphs : 

\begin {enumerate}
 \item 
Let us consider the clusters of $\mathcal{L}_{g,{1}/{2}}$. Two loops belong to the same cluster if they are connected by a finite chain of loops in $\mathcal{L}_{g,{1}/{2}}$, where any two consecutive loops
intersect each other. If $\Sigma$ is a simply connected domain of $\mathbb{C}$, the outer boundaries of
outermost clusters, i.e. not surrounded by an other cluster, are distributed like the $\hbox{CLE}_{4}$ 
conformal loop ensemble (\cite{SheffieldWerner2012CLE}).

\begin{conjecture}
\label{ConjMetricRiemSurf}
There is a metric on the clusters of $\mathcal{L}_{g,{1}/{2}}$ inside $\Sigma$, measurable
with respect to $\mathcal{L}_{g,{1}/{2}}$, which is obtained as the limit of the  pseudo-metric on metric graphs approximating $\Sigma$. The metric does not depend on the time-parametrisation of loops and on the particular choice of $g$ compatible with the complex structure of $\Sigma$ (in the case where $\Sigma$ is a simply connected domain of $\mathbb{C}$, the metric is the one given by the conformal invariant growth mechanism inside $\hbox{CLE}_{4}$). As in Proposition \ref{thmDist2Sets}, the square of the distance between
$B_{1}(\Sigma)$ and $B_{2}(\Sigma)$ induced by the metric on clusters of $\mathcal{L}_{g,{1}/{2}}$
is an exponential random variable, and its mean is $2  \operatorname{ED}(B_{1}(\Sigma),B_{2}(\Sigma))$. 
\end{conjecture}

\item
We consider now the case where the field $\phi_{\Sigma}$ has constant boundary condition $h_{2}$ on $B_{2}(\Sigma)$,
where it is not necessarily constant on $B_{1}(\Sigma)$, and has free boundary condition on $B_{3}(\Sigma)$.
Some regularity condition is needed for the boundary condition $h_1$ on $B_{1}(\Sigma)$. 
We for instance assume that $h_1$ is continuous and bounded on $B_1 (\Sigma)$
(but more general setups are also possible). 

We can consider the excursion measure $\mu$ of Brownian motion in $\Sigma$ (reflected on $B_3$ and killed upon hitting $B_1$ and $B_2$), and restrict it to the set of excursions that start on $B_1$ and end on $B_2$. This set of excursions has finite mass, so that it is possible to renormalise $\mu$ into a probability measure, and to consider the 
expected value of $h$ at the starting point of this excursion (according to this probability measure). We call this quantity $m$ (and note that it is 
exactly the continuous 2D counterpart of the quantity $m$ defined just before Proposition \ref {ThmReffLocSet}). Note that when $h_1$ is constant, then $m$ is equal to that constant. 

We assume that $a < \min_{B_1 (\Sigma)} h_1$.
We consider the local set $\Lambda_{a}^{\Sigma}$ for $\phi_{\Sigma}$, with the following properties:
\begin{itemize}
\item $\Lambda_{a}^{\Sigma}$ is compact,
\item $\Lambda_{a}^{\Sigma}$ contains $B_{1}(\Sigma)$,
\item $\Lambda_{a}^{\Sigma}$ is measurable with respect the GFF $\phi_{\Sigma}$,
\item Conditionally on $\Lambda_{a}^{\Sigma}$ and $\phi_{\Sigma}$ on $\Lambda_{a}^{\Sigma}$, 
the restriction of $\phi_{\Sigma}$ to $\Sigma\setminus \Lambda_{a}^{\Sigma}$ is distributed like a GFF with boundary condition $a$ on $\Lambda_{a}^{\Sigma}$, $h_{2}$ on $B_{2}(\Sigma)\setminus \Lambda_{a}^{\Sigma}$
and free on $B_{3}(\Sigma)\setminus \Lambda_{a}^{\Sigma}$,
\item $\Lambda_{a}^{\Sigma}$ is minimal for the above properties, that is to say any other random set satisfying the above four properties a.s. contains $\Lambda_{a}^{\Sigma}$.
\end{itemize}
$\Lambda_{a}^{\Sigma}$ is a first {passage} set, analogous to 
$\widetilde{\Lambda}_{a}$ on the metric graph, where $B_{1}(\Sigma)$ plays the role of $\widehat{A}$. 
$\Lambda_{a}^{\Sigma}$ can be for instance constructed out of a metric graph approximation of $\Sigma$. 

Proposition~\ref{ThmReffLocSet}  (and the comments that follow it) leads to the following: 
\begin{conjecture}
\label{ConjEL}
The extremal distance $\operatorname{ED}(\Lambda_{a}^{\Sigma},B_{2}(\Sigma))$ is distributed like the {\em last} visit time of the level $a$ by a Brownian bridge of length
$\operatorname{ED}(B_{1}(\Sigma),B_{2}(\Sigma))$ from $h_2$ to $m$ (this last time is defined to be $0$ when the bridge does not hit $a$ at all). 
\end{conjecture}

\item 
We finally consider the case there $\phi_{\Sigma}$ has some sufficiently regular boundary condition $h$ on $B_{1}(\Sigma)\cup B_{2}(\Sigma)$ and
free boundary condition on $B_{3}(\Sigma)$. Let $z_{0}\in \Sigma\setminus B(\Sigma)$.
We consider the local set $\Lambda_{a}^{\Sigma}(z_{0})$ for $\phi_{\Sigma}$, with the following properties:
\begin{itemize}
\item $\Lambda_{a}^{\Sigma}(z_{0})$ is compact,
\item $\Lambda_{a}^{\Sigma}(z_{0})$ contains $B_{1}(\Sigma)\cup B_{2}(\Sigma)$,
\item $\Sigma\setminus \Lambda_{a}^{\Sigma}(z_{0})$ is connected and contains $z_{0}$,
\item $\Lambda_{a}^{\Sigma}(z_{0})$ is measurable with respect the GFF $\phi_{\Sigma}$,
\item conditionally on $\Lambda_{a}^{\Sigma}(z_{0})$ and $\phi_{\Sigma}$ on $\Lambda_{a}^{\Sigma}(z_{0})$, 
the restriction of $\phi_{\Sigma}$ to $\Sigma\setminus \Lambda_{a}^{\Sigma}(z_{0})$ is distributed like a GFF with boundary condition $a$ on $\Lambda_{a}^{\Sigma}(z_{0})$
and free on $B_{3}(\Sigma)\setminus \Lambda_{a}^{\Sigma}(z_{0})$,
\item $\Lambda_{a}^{\Sigma}(z_{0})$ is minimal for the above properties, that is to say any other random set satisfying the above five properties a.s. contains $\Lambda_{a}^{\Sigma}(z_{0})$.
\end{itemize}
Let $G_{\Sigma}(z_{1},z_{2})$ denote the Green function of the Laplacian on $\Sigma$ with zero Dirichlet
boundary condition on $B_{1}(\Sigma)\cup B_{2}(\Sigma)$ and zero Neumann boundary condition on
$B_{3}(\Sigma)$. Let $G_{\Sigma\setminus \Lambda_{a}^{\Sigma}(z_{0})}(z_{1},z_{2})$ denote the Green function of the Laplacian on $\Sigma\setminus \Lambda_{a}^{\Sigma}(z_{0})$ with zero Dirichlet
boundary condition on $\Lambda_{a}^{\Sigma}(z_{0})$ and zero Neumann boundary condition on
$B_{3}(\Sigma)\setminus \Lambda_{a}^{\Sigma}(z_{0})$. For any 
$z_{1}\in \Sigma\setminus \Lambda_{a}^{\Sigma}(z_{0})$, the function
\begin{displaymath}
z_{2}\mapsto G_{\Sigma}(z_{1},z_{2})-G_{\Sigma\setminus \Lambda_{a}^{\Sigma}(z_{0})}(z_{1},z_{2})
\end{displaymath}
is harmonic on $\Sigma\setminus \Lambda_{a}^{\Sigma}(z_{0})$ and thus has a continuous extension at 
$z_{2}=z_{1}$. Moreover
\begin{displaymath}
G_{\Sigma}(z_{1},z_{1})-G_{\Sigma\setminus \Lambda_{a}^{\Sigma}(z_{0})}(z_{1},z_{1})
\end{displaymath}
is the limit of differences of effective resistances in an approximation of $\Sigma$ by metric graphs.
Let $m$ be the value at $z_{0}$ of the harmonic extension of $h$ (with zero Neumann condition on 
$B_{3}(\Sigma)$). If $h$ is constant on $B_{1}(\Sigma)\cup B_{2}(\Sigma)$, then $m$ is that constant.
Corollary \ref{CorDistrib1ptandHitting} leads to the following.

\begin{conjecture}
\label{ConjGreenFunc}
The random variable  $
G_{\Sigma}(z_{0},z_{0})-G_{\Sigma\setminus \Lambda_{a}^{\Sigma}(z_{0})}(z_{0},z_{0})$ 
is distributed like the first hitting time of the level $a$ by a standard Brownian motion starting from
$m$.
\end{conjecture}
\end {enumerate}

\section*{Appendix}

Here we provide some details about the rather brute-force It\^o formula computation of 
the semi-martingale decomposition of $\Psi_{\ell}(t)$ that proves Lemma \ref{LemMartingale}. We use the notation of the beginning of Section \ref{SecExplicit}.
To simplify the expressions, we will use the following notation:
\begin{displaymath}
r^{\ast}_{i}(t):=\hat{r}_{i}(t\wedge\tau_{\ell}),
\qquad
\tilde{\phi}^{\ast}_{i}(t):=\tilde{\phi}_{\hat{z}_{i}(\hat{r}_{i}(t\wedge \tau_{\ell}))},
\qquad
L^{\ast}(t):=L(t\wedge\tau_{\ell}),
\end{displaymath}
$$
C^{\ast}_{ij}(t):=C^{\rm eff}_{ij}(t\wedge\tau_{\ell}),
\qquad
C^{\ast}_{i\check{x}}(t):=C^{\rm eff}_{i\check{x}}(t\wedge\tau_{\ell}),
$$
\begin{displaymath}
S_{i\check{x}}(t):=\tilde{\phi}_{\hat{z}_{i}(\hat{r}_{i}(t\wedge \tau_{\ell}))}-h(\check{x}),
\qquad S^{+}_{i\check{x}}(t):=\vert\tilde{\phi}_{\hat{z}_{i}(\hat{r}_{i}(t\wedge \tau_{\ell}))}\vert +
\vert h(\check{x})\vert+\ell-L(t\wedge\tau_{\ell}),
\end{displaymath}
\begin{displaymath}
F(t)=\sum_{\substack{1\leq i\leq\hat{n}\\\check{x}\in\widecheck{A}}}
C^{\ast}_{i\check{x}}(t)(S^{+}_{i\check{x}}(t)^{2}-S_{i\check{x}}(t)^{2}).
\end{displaymath}
At each time $t$, the $\tilde{\phi}^{\ast}_{i}(t)$ for different values of $i$ have the same sign.
We will denote this common sign by $\sigma(t)\in \lbrace -1,+1\rbrace$. If all the 
$\tilde{\phi}^{\ast}_{i}(t)$ are zero, the value of $\sigma(t)$ does not matter.
Similarly, we will denote $\hbox {sgn}(h ( \widecheck A))\in \lbrace -1, +1\rbrace$ the common sign of $h$ on $\widecheck{A}$. If $h$ vanishes on $\widecheck{A}$, the value of the sign does not matter.

According to Lemma \ref{LemSDE}, $\tilde{\phi}^{\ast}_{i}(t)$ has the following semi-martingale decomposition:
$$ M^{\ast}_{i}(t)+
\sum_{j\neq i}\int_{0}^{t}
C^{\ast}_{ij}(s)(\tilde{\phi}^{\ast}_{j}(s)-\tilde{\phi}^{\ast}_{i}(s)) dr^{\ast}_{i}(s)
+\sum_{\check{x}\in\widecheck{A}}\int_{0}^{t}
C^{\ast}_{i\check{x}}(s)(h(\check{x})-\tilde{\phi}^{\ast}_{i}(s)) dr^{\ast}_{i}(s),
$$
where $M^{\ast}_{i}(t)$ is an $(\widehat{\mathcal{F}}_{t})_{t\geq 0}$-martingale. Moreover,
$\langle M_{i}^{\ast},M_{i}^{\ast}\rangle_{t}=r^{\ast}_{i}(t)$ and $
\langle M_{i}^{\ast},M_{j}^{\ast}\rangle_{t}=0$ for $i\neq j$.
Further, $d S_{i\check{x}}(t)= d \tilde{\phi}^{\ast}_{i}(t)$,
$d S_{i\check{x}}^{+}(t)=  \sigma(t) d\tilde{\phi}^{\ast}_{i}(t)$,
and
$d\langle S_{i\check{x}},S_{i\check{x}}\rangle_{t}=
d\langle S^{+}_{i\check{x}},S^{+}_{i\check{x}}\rangle_{t}
= d r^{\ast}_{i}(t).$
According to Lemma \ref{LemDerivPartFunc}, Equations \eqref{EqDerCond} and \eqref{EqDerCond2},
\begin{displaymath}
d C^{\ast}_{i\check{x}}(t) = C^{\ast}_{i\check{x}}(t)\sum_{j\neq i}C^{\ast}_{ij}(t) d r^{\ast}_{i}(t)
+C^{\ast}_{i\check{x}}(t)\sum_{\check{x}'\in\widecheck{A}}C^{\ast}_{i\check{x}'}(t) d r^{\ast}_{i}(t)
-\sum_{j\neq i}C^{\ast}_{ij}(t)C^{\ast}_{j\check{x}}(t) d r^{\ast}_{j}(t).
\end{displaymath}

Using It\^o's formula formula we get that
\begin{eqnarray*}
&&d(S^{+}_{i\check{x}}(t)^{2}-S_{i\check{x}}(t)^{2})=
2S^{+}_{i\check{x}}(t) dS^{+}_{i\check{x}}(t)+ d\langle S^{+}_{i\check{x}},S^{+}_{i\check{x}}\rangle_{t}
-2S_{i\check{x}}(t) dS_{i\check{x}}(t)
- d\langle S_{i\check{x}},S_{i\check{x}}\rangle_{t}
\\
&&=
2(\sigma(t)S^{+}_{i\check{x}}(t)-S_{i\check{x}}(t)) d\tilde{\phi}^{\ast}_{i}(t)=
2\theta_{\check{x}}(t)d\tilde{\phi}^{\ast}_{i}(t)
\\
&&=
2\theta_{\check{x}}(t)\bigg(
d M^{\ast}_{i}(t)+
\sum_{j\neq i}
C^{\ast}_{ij}(t)(\tilde{\phi}^{\ast}_{j}(t)-\tilde{\phi}^{\ast}_{i}(t)) dr^{\ast}_{i}(t)
+\sum_{\check{x}'\in\widecheck{A}}
C^{\ast}_{i\check{x}'}(t)(h(\check{x}')-\tilde{\phi}^{\ast}_{i}(t)) dr^{\ast}_{i}(t)
\bigg),
\end{eqnarray*}
where $\theta_{\check{x}}(t)=(\sigma(t)\vert h(\check{x})\vert + h(\check{x})+\sigma(t)(\ell-L^{\ast}(t)))$.
Further
\begin{eqnarray*}
d F(t)&=&
\sum_{\substack{1\leq i\leq\hat{n}\\\check{x}\in\widecheck{A}}}(
2\theta_{\check{x}}(t)C^{\ast}_{i\check{x}}(t)d\tilde{\phi}^{\ast}_{i}(t)+
(S^{+}_{i\check{x}}(t)^{2}-S_{i\check{x}}(t)^{2})
dC^{\ast}_{i\check{x}}(t))
\\&=&
2\sum_{\substack{1\leq i\leq\hat{n}\\\check{x}\in\widecheck{A}}}
\theta_{\check{x}}(t)C^{\ast}_{i\check{x}}(t)d\tilde{\phi}^{\ast}_{i}(t)
+\sum_{\substack{i\neq j\\\check{x}\in\widecheck{A}}}
C^{\ast}_{i\check{x}}(t)C^{\ast}_{ij}(t)
(S^{+}_{i\check{x}}(t)^{2}-S_{i\check{x}}(t)^{2}-S^{+}_{j\check{x}}(t)^{2}+S_{j\check{x}}(t)^{2})
dr_{i}^{\ast}(t)
\\&&+
\sum_{\substack{1\leq i\leq\hat{n}\\\check{x},\check{x}'\in\widecheck{A}}}
(S^{+}_{i\check{x}}(t)^{2}-S_{i\check{x}}(t)^{2})
C^{\ast}_{i\check{x}}(t)C^{\ast}_{i\check{x}'}(t)dr_{i}^{\ast}(t)
\\&=&
2\sum_{\substack{1\leq i\leq\hat{n}\\\check{x}\in\widecheck{A}}}
\theta_{\check{x}}(t)C^{\ast}_{i\check{x}}(t)d\tilde{\phi}^{\ast}_{i}(t)
-2\sum_{\substack{i\neq j\\\check{x}\in\widecheck{A}}}\theta_{\check{x}}(t)
C^{\ast}_{i\check{x}}(t)C^{\ast}_{ij}(t)(\tilde{\phi}_{j}(t)-\tilde{\phi}_{i}(t))d r^{\ast}_{i}(t)
\\&&+
\sum_{\substack{1\leq i\leq\hat{n}\\\check{x},\check{x}'\in\widecheck{A}}}
(\theta_{\check{x}}(t)^{2}+2\theta_{\check{x}}(t)(\tilde{\phi}_{i}(t)-h(\check{x})))
C^{\ast}_{i\check{x}}(t)C^{\ast}_{i\check{x}'}(t)dr_{i}^{\ast}(t)
\\&=&
2\sum_{\substack{1\leq i\leq\hat{n}\\\check{x}\in\widecheck{A}}}
\theta_{\check{x}}(t)C^{\ast}_{i\check{x}}(t) dM_{i}^{\ast}(t)
+\sum_{\substack{1\leq i\leq\hat{n}\\\check{x},\check{x}'\in\widecheck{A}}}
\theta_{\check{x}}(t)^{2}C^{\ast}_{i\check{x}}(t)C^{\ast}_{i\check{x}'}(t)dr_{i}^{\ast}(t)
\\&&+
2\sum_{\substack{1\leq i\leq\hat{n}\\\check{x},\check{x}'\in\widecheck{A}}}
\theta_{\check{x}}(t)(h(\check{x}')-h(\check{x}))
C^{\ast}_{i\check{x}}(t)C^{\ast}_{i\check{x}'}(t)dr_{i}^{\ast}(t)
\\&=&
2\sum_{\substack{1\leq i\leq\hat{n}\\\check{x}\in\widecheck{A}}}
\theta_{\check{x}}(t)C^{\ast}_{i\check{x}}(t) dM_{i}^{\ast}(t)
+\sum_{\substack{1\leq i\leq\hat{n}\\\check{x},\check{x}'\in\widecheck{A}}}
\theta_{\check{x}}(t)^{2}C^{\ast}_{i\check{x}}(t)C^{\ast}_{i\check{x}'}(t)dr_{i}^{\ast}(t)
\\&&-(1+\sigma(t)\hbox{sgn}(h ( \widecheck A)))
\sum_{\substack{1\leq i\leq\hat{n}\\\check{x},\check{x}'\in\widecheck{A}}}
(h(\check{x}')-h(\check{x}))^{2}
C^{\ast}_{i\check{x}}(t)C^{\ast}_{i\check{x}'}(t)dr_{i}^{\ast}(t),
\end{eqnarray*}

\begin{eqnarray*}
d \langle F,F\rangle_{t}&=&
4\sum_{\substack{1\leq i\leq\hat{n}\\\check{x},\check{x}'\in\widecheck{A}}}
\theta_{\check{x}}(t)\theta_{\check{x}'}(t)C^{\ast}_{i\check{x}}(t)C^{\ast}_{i\check{x}'}(t)dr_{i}^{\ast}(t)
\\&=&
4\sum_{\substack{1\leq i\leq\hat{n}\\\check{x},\check{x}'\in\widecheck{A}}}
\theta_{\check{x}}(t)^{2}C^{\ast}_{i\check{x}}(t)C^{\ast}_{i\check{x}'}(t)dr_{i}^{\ast}(t)
+
4\sum_{\substack{1\leq i\leq\hat{n}\\\check{x},\check{x}'\in\widecheck{A}}}
\theta_{\check{x}}(t)(\theta_{\check{x}'}(t)-\theta_{\check{x}}(t))
C^{\ast}_{i\check{x}}(t)C^{\ast}_{i\check{x}'}(t)dr_{i}^{\ast}(t)
\\&=&
4\sum_{\substack{1\leq i\leq\hat{n}\\\check{x},\check{x}'\in\widecheck{A}}}
\theta_{\check{x}}(t)^{2}C^{\ast}_{i\check{x}}(t)C^{\ast}_{i\check{x}'}(t)dr_{i}^{\ast}(t)
-
2\sum_{\substack{1\leq i\leq\hat{n}\\\check{x},\check{x}'\in\widecheck{A}}}
(\theta_{\check{x}'}(t)-\theta_{\check{x}}(t))^{2}
C^{\ast}_{i\check{x}}(t)C^{\ast}_{i\check{x}'}(t)dr_{i}^{\ast}(t)
\\&=&
4\sum_{\substack{1\leq i\leq\hat{n}\\\check{x},\check{x}'\in\widecheck{A}}}
\theta_{\check{x}}(t)^{2}C^{\ast}_{i\check{x}}(t)C^{\ast}_{i\check{x}'}(t)dr_{i}^{\ast}(t)
\\&&-4(1+\sigma(t)\hbox{sgn}(h ( \widecheck A)))
\sum_{\substack{1\leq i\leq\hat{n}\\\check{x},\check{x}'\in\widecheck{A}}}
(h(\check{x}')-h(\check{x}))^{2}
C^{\ast}_{i\check{x}}(t)C^{\ast}_{i\check{x}'}(t)dr_{i}^{\ast}(t).
\end{eqnarray*}

Finally, 
\begin{displaymath}
\dfrac{d\Psi_{\ell}(t)}{\Psi_{\ell}(t)} =
-\dfrac{1}{2} dF(t)+\dfrac{1}{8} d \langle F,F\rangle_{t}
= -\sum_{\substack{1\leq i\leq\hat{n}\\\check{x}\in\widecheck{A}}}
\theta_{\check{x}}(t)C^{\ast}_{i\check{x}}(t) dM_{i}^{\ast}(t).
\end{displaymath}
It follows that $\Psi_{\ell}(t)$ is a local martingale.

\section*{Acknowledgements}

TL acknowledges the support of Dr. Max Rössler, the Walter Haefner
Foundation and the ETH Zurich Foundation. WW acknowledges the support of the SNF grant
SNF-155922. The authors are also part of the NCCR Swissmap of the SNF. They also thank the referees for their comments. 

\bibliographystyle{plain}

\end{document}